\documentclass[12pt,oneside,english,reqno,twoside]{amsart}
\usepackage{lmodern}

\usepackage[T1]{fontenc}
\usepackage[latin9]{luainputenc}
\usepackage{geometry}
\usepackage{amstext}
\usepackage{amsthm}
\usepackage{amssymb}

\makeatletter
\numberwithin{equation}{section}
\numberwithin{figure}{section}
\theoremstyle{plain}
\newtheorem{thm}{\protect\theoremname}[section]
  \theoremstyle{definition}
  \newtheorem{defn}[thm]{\protect\definitionname}
  \theoremstyle{plain}
  \newtheorem{lem}[thm]{\protect\lemmaname}
  \theoremstyle{remark}
  \newtheorem{rem}[thm]{\protect\remarkname}

\usepackage{versions}

\DeclareMathOperator*{\esssup}{ess\,sup}

\DeclareMathOperator*{\esssinf}{ess\,inf}

\DeclareMathOperator*{\esssliminf}{ess\,lim\,inf}

\def\Xint#1{\mathchoice
{\XXint\displaystyle\textstyle{#1}}%
{\XXint\textstyle\scriptstyle{#1}}%
{\XXint\scriptstyle\scriptscriptstyle{#1}}%
{\XXint\scriptscriptstyle\scriptscriptstyle{#1}}%
\!\int}
\def\XXint#1#2#3{{\setbox0=\hbox{$#1{#2#3}{\int}$ }
\vcenter{\hbox{$#2#3$ }}\kern-.6\wd0}}

\def\dashint{\Xint-}

\usepackage{hyperref}
\newcommand{\Addresses}{{
  \bigskip
  \footnotesize
  \textsc{Jarkko Siltakoski,  Department of Mathematics and Statistics, P.O.Box 35, FIN-40014, University of Jyväskylä, Finland}\par\nopagebreak
  \textit{E-mail address}: \href{mailto:jarkko.j.m.siltakoski@student.jyu.fi}{jarkko.j.m.siltakoski@student.jyu.fi}
}}

\makeatother

\usepackage{babel}
  \providecommand{\definitionname}{Definition}
  \providecommand{\lemmaname}{Lemma}
  \providecommand{\remarkname}{Remark}
\providecommand{\theoremname}{Theorem}

\begin{document}
\global\long\def\d{\,d}
\global\long\def\tr{\mathrm{tr}}
\global\long\def\supp{\operatorname{spt}}
\global\long\def\div{\operatorname{div}}
\global\long\def\osc{\operatorname{osc}}
\global\long\def\essup{\esssup}
\global\long\def\aint{\dashint}
\global\long\def\essinf{\esssinf}
\global\long\def\essliminf{\esssliminf}

\excludeversion{old}

\excludeversion{note}

\excludeversion{note2}


\title[Equivalence of viscosity and weak solutions]{Equivalence of viscosity and weak solutions for a $p$-parabolic
equation}

\author{Jarkko Siltakoski}
\begin{abstract}
We study the relationship of viscosity and weak solutions to the equation
\[
\smash{\partial_{t}u-\Delta_{p}u=f(Du)}
\]
where $p>1$ and $f\in C(\mathbb{R}^{N})$ satisfies suitable assumptions.
Our main result is that bounded viscosity supersolutions coincide
with bounded lower semicontinuous weak supersolutions. Moreover, we
prove the lower semicontinuity of weak supersolutions when $p\geq2$.
\end{abstract}

\maketitle

\section{Introduction}

A classical solution to a partial differential equation is a smooth
function that satisfies the equation pointwise. Since many equations
that appear in applications admit no such solutions, a more general
class of solutions is needed. One such class is the extensively studied
distributional weak solutions defined by integration by parts. Another
is the celebrated viscosity solutions based on generalized pointwise
derivatives. When both classes of solutions can be meaningfully defined,
it is naturally crucial that they coincide. This has been profusely
studied starting from \cite{Ishii95}. In \cite{equivalence_plaplace}
the equivalence of solutions was proved for the parabolic $p$-Laplacian.
The objective of the present work is to prove this equivalence in
a different way while also allowing the equation to depend on a first-order
term. To the best of our knowledge, the proof is new even in the homogeneous
case, at least when $\smash{1<p<2}$.

More precisely, we study the parabolic equation
\begin{equation}
\partial_{t}u-\Delta_{p}u=f(Du)\label{eq:p-para f}
\end{equation}
where $\smash{1<p<\infty}$ and $\smash{f\in C(\mathbb{R}^{N})}$
satisfies a certain growth condition, for details see Section 2. We
show that bounded viscosity supersolutions to (\ref{eq:p-para f})
coincide with bounded lower semicontinuous weak supersolutions. Moreover,
we prove the lower semicontinuity of weak supersolutions in the range
$p\geq2$ under slightly stronger \mbox{assumptions on $f$}.

To show that viscosity supersolutions are weak supersolutions, we
apply the technique introduced by Julin and Juutinen \cite{newequivalence}.
In contrast to \cite{equivalence_plaplace}, we do not employ the
uniqueness machinery of viscosity solutions. Instead, our strategy
is to approximate a viscosity supersolution $u$ by its inf-convolution
$\smash{u_{\varepsilon}}$. It is straightforward to show that $\smash{u_{\varepsilon}}$
is still a viscosity supersolution in a smaller domain. This and the
pointwise properties of the inf-convolution imply that $\smash{u_{\varepsilon}}$
is also a weak supersolution in the smaller domain. Furthermore, it
follows from Caccioppoli's estimates that $\smash{u_{\varepsilon}}$
converges to $u$ in a suitable Sobolev space. It then remains to
pass to the limit to see that $u$ is a weak supersolution.

To show that weak supersolutions are viscosity supersolutions, we
apply the argument from \cite{equivalence_plaplace} that is based
on the comparison principle of weak solutions. However, we could not
find a reference for comparison principle for the equation (\ref{eq:p-para f}).
Therefore we give a detailed proof of such a result.

To prove the lower semicontinuity of weak supersolutions, we adapt
the strategy of \cite{kuusi09}. First we prove estimates for the
essential supremum of a subsolution using the Moser's iteration technique.
Then we use those estimates to deduce that a supersolution is lower
semicontinuous at its Lebesgue points.

The equivalence of viscosity and weak solutions for the $p$-Laplace
equation and its parabolic version was first proven in \cite{equivalence_plaplace}.
A different proof in the elliptic case was found in \cite{newequivalence}.
Recently the equivalence of solutions has been studied for various
equations. These include the normalized $p$-Poisson equation \cite{OptimalC1},
a non-homogeneous $p$-Laplace equation \cite{chilepaper} and the
normalized $\smash{p(x)}$-Laplace equation \cite{siltakoski18}.
Moreover, in \cite{ParvVaz} the equivalence is shown for the radial
solutions of a parabolic equation. We also mention that an unpublished
version of \cite{lindqvist12reg} applies \cite{newequivalence} to
sketch the equivalence of solutions to (\ref{eq:p-para f}) in the
homogeneous case when $p\geq2$.

Comparison principles for quasilinear parabolic equations have been
studied by several authors. In \cite{junning90} comparison is proven
for $\smash{\partial_{t}u-\Delta_{p}u+f(u,x,t)=0}$ when $p>2$ and
$\smash{f}$ is a continuous function such that $\smash{\left|f(u,x,t)\right|\leq g(u)}$
for some $\smash{g\in C^{1}}$. The homogeneous case for the $p$-parabolic
equation is considered also in \cite{kilpel=0000E4inenLindqvist96}
and the general equation $\smash{\partial_{t}u-\div\mathcal{A}(x,t,Du)=0}$
in \cite{korteKuusiParv10}. Equations with gradient terms are studied
for example in \cite{Attouchi12}, where comparison principle is shown
for the equation $\smash{\partial_{t}u-\Delta_{p}u-\left|Du\right|^{\beta}=0}$
when $\smash{p>2}$ and $\smash{\beta>p-1}$. In the recent papers
\cite{bobkovTakac14,bobkovTakac18}, both positive results and counter
examples are provided for the comparison, strong comparison and maximum
principles for the equation $\smash{\partial_{t}u-\Delta_{p}u-\lambda\left|u\right|^{p-2}u-f(x,t)=0}$.
Furthermore, according to \cite{benediktGirgKotrlaTakac16}, the equation
$\smash{\partial_{t}u-\Delta_{p}u=q(x)\left|u\right|^{\alpha}}$ can
admit multiple solutions with zero boundary values when $\smash{0<\alpha<1}$.

The paper is organized as follows. Section 2 contains the precise
definitions of weak and viscosity solutions. In Section 3 we show
that weak supersolutions are viscosity supersolutions, and the converse
is shown in Section 4. Finally, the lower semicontinuity of weak supersolutions
is considered in Section 5.

\section{Preliminaries }

The symbols $\Xi$ and $\Omega$ are reserved for bounded domains
in $\smash{\mathbb{R}^{N}\times\mathbb{R}}$ and $\smash{\mathbb{R}^{N}}$,
respectively. For $\smash{t_{1}<t_{2}}$, we define the cylinder $\smash{\Omega_{t_{1},t_{2}}:=\Omega\times(t_{1},t_{2})}$
and its \textit{parabolic boundary} $\smash{\partial_{p}\Omega_{t_{1},t_{2}}:=(\overline{\Omega}\times\left\{ t_{1}\right\} )\cup\left(\partial\Omega\times(t_{1},t_{2}]\right)}$.
Moreover, for $T>0$ we set $\smash{\Omega_{T}:=\Omega_{0,T}}$.

The \textit{Sobolev space} $\smash{W^{1,p}(\Omega)}$ contains the
functions $\smash{u\in L^{p}(\Omega)}$ for which the distributional
gradient $\smash{Du}$ exists and belongs in $\smash{L^{p}(\Omega)}$.
It is equipped with the norm 
\[
\left\Vert u\right\Vert _{W^{1,p}(\Omega)}:=\left\Vert u\right\Vert _{L^{p}(\Omega)}+\left\Vert Du\right\Vert _{L^{p}(\Omega)}.
\]
A Lebesgue measurable function $\smash{u:\Omega_{t_{1},t_{2}}\rightarrow\mathbb{R}}$
belongs to the \textit{parabolic Sobolev space} $\smash{L^{p}(t_{1},t_{2};W^{1,p}(\Omega))}$
if $\smash{u(\cdot,t)\in W^{1,p}(\Omega)}$ for almost every $\smash{t\in(t_{1},t_{2})}$
and the norm 
\[
\left(\int_{\Omega_{t_{1},t_{2}}}\left|u\right|^{p}+\left|Du\right|^{p}\d z\right)^{\frac{1}{p}}
\]
is finite. By $dz$ we mean integration with respect to space and
time variables, i.e. $dz=dx\d t$. Integral average is denoted by
\[
\aint_{\Omega_{T}}u\d z:=\frac{1}{\left|\Omega_{T}\right|}\int_{\Omega_{T}}u\d z.
\]

\subsection*{Growth condition}

Unless otherwise stated, the function $f\in C(\mathbb{R}^{N})$ is
assumed to satisfy the growth condition
\begin{equation}
\left|f(\xi)\right|\leq C_{f}(1+\left|\xi\right|^{\beta})\quad\text{for all }\xi\in\mathbb{R}^{N},\tag{{G1}}\label{eq:gcnd}
\end{equation}
where $C_{f}>0$ and $1\leq\beta<p$. 

\vspace{0bp}

\begin{defn}[Weak solution]
A function $u:\Xi\rightarrow\mathbb{R}$ is a \textit{weak supersolution}
to (\ref{eq:p-para f}) in $\Xi$ if $u\in L^{p}(t_{1},t_{2};W^{1,p}(\Omega))$
whenever $\Omega_{t_{1},t_{2}}\Subset\Xi$, and
\begin{equation}
\int_{\Xi}-u\partial_{t}\varphi+\left|Du\right|^{p-2}Du\cdot D\varphi-\varphi f(Du)\d z\geq0\label{eq:p-para-f weak eq}
\end{equation}
for all non-negative \textit{test functions} $\varphi\in C_{0}^{\infty}(\Omega_{t_{1},t_{2}})$.
For \textit{weak subsolutions} the inequality is reversed and a function
is a \textit{weak solution} if it is both super- and subsolution.
\end{defn}
To define viscosity solutions to (\ref{eq:p-para f}), we set for
all $\varphi\in C^{2}$ with $D\varphi\not=0$
\[
\Delta_{p}\varphi:=\left|D\varphi\right|^{p-2}\left(\Delta\varphi+\frac{p-2}{\left|D\varphi\right|^{2}}\left\langle D^{2}\varphi D\varphi,D\varphi\right\rangle \right).
\]

\begin{defn}[Viscosity solution]
A lower semicontinuous and bounded function $u:\Xi\rightarrow\mathbb{R}$
is a \textit{viscosity supersolution} to (\ref{eq:p-para f}) in $\Xi$
if whenever $\varphi\in C^{2}(\Xi)$ and $(x_{0},t_{0})\in\Xi$ are
such that
\[
\begin{cases}
\varphi(x_{0},t_{0})=u(x_{0},t_{0}),\\
\varphi(x,t)<u(x,t) & \text{when }(x,t)\not=(x_{0},t_{0}),\\
D\varphi(x,t)\not=0 & \text{when }x\not=x_{0},
\end{cases}
\]
then
\[
\limsup_{\substack{\substack{(x,t)\rightarrow(x_{0},t_{0})\\
x\not=x_{0}
}
}
}\left(\partial_{t}\varphi(x,t)-\Delta_{p}\varphi(x,t)-f(D\varphi(x,t))\right)\geq0.
\]
An upper semicontinuous and bounded function $u:\Xi\rightarrow\mathbb{R}$
is a \textit{viscosity subsolution} to (\ref{eq:p-para f}) in $\Xi$
if whenever $\varphi\in C^{2}(\Xi)$ and $(x_{0},t_{0})\in\Xi$ are
such that
\[
\begin{cases}
\varphi(x_{0},t_{0})=u(x_{0},t_{0}),\\
\varphi(x,t)>u(x,t) & \text{when }(x,t)\not=(x_{0},t_{0}),\\
D\varphi(x,t)\not=0 & \text{when }x\not=x_{0,}
\end{cases}
\]
then
\[
\liminf_{\substack{\substack{(x,t)\rightarrow(x_{0},t_{0})\\
x\not=x_{0}
}
}
}\left(\partial_{t}\varphi(x,t)-\Delta_{p}\varphi(x,t)-f(D\varphi(x,t))\right)\leq0.
\]
A function that is both viscosity sub- and supersolution is a \textit{viscosity
solution}.
\end{defn}
If a function $\varphi$ is like in the definition of viscosity supersolution,
we say that $\varphi$ \textit{touches $u$ from below at} $\smash{(x_{0},t_{0})}$.
The limit supremum in the definition is needed because the operator
$\smash{\Delta_{p}}$ is singular when $1<p<2$. When $p\geq2$, the
operator is degenerate and the limit supremum disappears.

\section{Weak solutions are viscosity solutions}

We show that bounded, lower semicontinuous weak supersolutions to
(\ref{eq:p-para f}) are viscosity supersolutions when $\smash{1<p<\infty}$
and $\smash{f\in C(\mathbb{R}^{N})}$ satisfies the growth condition
(\ref{eq:gcnd}). One way to prove this kind of results is by applying
the comparison principle \cite{equivalence_plaplace}. However, we
could not find the comparison principle for the equation (\ref{eq:p-para f})
in the literature and therefore we prove it first. To this end, we
first prove comparison Lemmas \ref{lem:comparison lemma 1<p<2} and
\ref{lem:comparison lemma p>2} for locally Lipschitz continuous $\smash{f}$.
The local Lipschitz continuity allows us to absorb the first-order
terms into the terms that appear due to the $p$-Laplacian, see Step
2 in proof of Lemma \ref{lem:comparison lemma 1<p<2}. To deal with
general $\smash{f}$, we take a locally Lipschitz continuous approximant
$\smash{f_{\delta}}$ such that $\smash{\left\Vert f-f_{\delta}\right\Vert _{L^{\infty}(\mathbb{R}^{N})}<\delta/4T}$.
Then for sub- and supersolutions $u$ and $v$, we consider the functions
\[
u_{\delta}:=u-\frac{\delta}{T-t/2}\quad\text{and}\quad v_{\delta}:=v+\frac{\delta}{T-t/2}.
\]
These functions will be sub- and supersolutions to (\ref{eq:p-para f})
where $\smash{f}$ is replaced by $\smash{f_{\delta}}$. Since $\smash{f_{\delta}}$
is locally Lipschitz continuous, it follows from the Lemmas \ref{lem:comparison lemma 1<p<2}
and \ref{lem:comparison lemma p>2} that $\smash{u_{\delta}\leq v_{\delta}}$.
Letting $\smash{\delta\rightarrow0}$ then yields that $\smash{u\leq v}$.

For similar comparison results, see \cite[Proposition 2.1]{Attouchi12}
and \cite{junning90}. See also Chapters 3.5 and 3.6 in \cite{pucciSerrin07}
for the elliptic case. A minor difference in our results is that instead
of requiring that both the subsolution and the supersolution have
uniformly bounded gradients, we only require this for the subsolution.

To prove the comparison principle, we need to use a test function
that depends on the supersolution itself. However, supersolutions
do not necessarily have a time derivative. One way to deal with this
is to use mollifications in the time direction. For a compactly supported
$\smash{\varphi\in L^{p}(\Omega_{T})}$ we define its \textit{time-mollification}
by
\[
\varphi^{\epsilon}(x,t)=\int_{\mathbb{R}}\phi(x,t-s)\rho_{\epsilon}(s)\d s,
\]
where $\smash{\rho_{\epsilon}}$ is a standard mollifier whose support
is contained in $(-\epsilon,\epsilon)$. Then $\smash{\varphi^{\epsilon}}$
has time derivative and $\smash{\varphi^{\epsilon}\rightarrow\varphi}$
in $\smash{L^{p}(\Omega_{T})}$. Furthermore, the time-mollification
of a supersolution satisfies a reguralized equation in the sense of
the following lemma.
\begin{lem}
\label{lem:time conv lemma} Let $v\in L^{\infty}(\Omega_{T})$ be
a weak supersolution (subsolution) to (\ref{eq:p-para f}) in $\Omega_{T}$.
Then we have
\begin{equation}
\int_{\Omega_{T}}-v^{\epsilon}\partial_{t}\varphi+\left(\left|Dv\right|^{p-2}Dv\right)^{\epsilon}\cdot D\varphi-\varphi\left(f(Dv)\right)^{\epsilon}\d z\geq(\leq)\thinspace0\label{eq:p-para f reg}
\end{equation}
for all $\smash{\varphi\in W^{1,p}(\Omega_{T})\cap L^{\infty}(\Omega_{T})}$
with compact support in $\smash{\Omega_{T}}$. Moreover, if the stronger
growth condition (\ref{eq:gcnd2}) holds, then the assumption $\smash{\varphi\in L^{\infty}(\Omega_{T})}$
is not needed.
\end{lem}
If $\smash{\varphi}$ is smooth, then testing the weak formulation
of (\ref{eq:p-para f}) with $\smash{\varphi^{\epsilon}}$, changing
variables and using Fubini's theorem yields (\ref{eq:p-para f reg}).
The general case follows by approximating $\smash{\varphi}$ in $\smash{W^{1,p}(\Omega_{T})}$
with the standard mollification. We omit the details.
\begin{lem}
\label{lem:comparison lemma 1<p<2}Let $\smash{1<p<2}$ and let $\smash{f}$
be locally Lipschitz. Let $\smash{u}$, $\smash{v\in L^{\infty}(\Omega_{T})}$
respectively be weak sub- and supersolutions to (\ref{eq:p-para f})
in $\smash{\Omega_{T}}$. Assume that for all $\smash{(x_{0},t_{0})\in\partial_{p}\Omega_{T}}$
\[
\limsup_{(x,t)\rightarrow(x_{0},t_{0})}u(x,t)\leq\liminf_{(x,t)\rightarrow(x_{0},t_{0})}v(x,t).
\]
Suppose also that $Du\in L^{\infty}(\Omega_{T})$. Then $u\leq v$
a.e.\ in $\Omega_{T}$.
\end{lem}
\begin{proof}
\textbf{(Step 1)}\ Let $l>0$ and set $\smash{w:=(u-v-l)_{+}}$.
Let also $\smash{s\in(0,T)}$. We want to use $\smash{w\cdot\chi_{[0,s]}}$
as a test function, but since it is not smooth, we must perform mollifications.
Let $\smash{h>0}$ and define
\[
\varphi:=\eta\left(\left(u-v-l\right){}^{\epsilon}\right)_{+},
\]
 where
\[
\eta(t)=\begin{cases}
1, & t\in(0,s-h],\\
(-t+s+h)/2h, & t\in(s-h,s+h),\\
0, & t\in[s+h,T).
\end{cases}
\]
The function $\varphi$ is compactly supported and belongs in $W^{1,p}(\Omega_{T})$.
Therefore by Lemma \ref{lem:time conv lemma} we have
\begin{align}
\int_{\Omega_{T}} & -(u-v)^{\epsilon}\partial_{t}\varphi\d z\nonumber \\
\leq & \int_{\Omega_{T}}\left(\left(\left|Dv\right|^{p-2}Dv\right)^{\epsilon}-\left(\left|Du\right|^{p-2}Du\right)^{\epsilon}\right)\cdot D\varphi+\varphi\left(f(Du)^{\epsilon}-f(Dv)^{\epsilon}\right)\d z.\label{eq:comparison pre main est 1<p<2}
\end{align}
We use the linearity of convolution and integration by parts to eliminate
the time derivative. We obtain
\begin{align*}
\int_{\Omega_{T}}- & (u-v)^{\epsilon}\partial_{t}\varphi\d z\\
= & -\int_{\Omega_{T}}(u-v)^{\epsilon}\left((u-v-l)^{\epsilon}\right)_{+}\partial_{t}\eta+\eta(u-v)^{\epsilon}\partial_{t}\left((u-v-l)^{\epsilon}\right)_{+}\d z\\
= & -\int_{\Omega_{T}}(u-v-l)^{\epsilon}((u-v-l)^{\epsilon})_{+}\partial_{t}\eta+l\left((u-v-l)^{\epsilon}\right)_{+}\partial_{t}\eta\\
 & \ \ \ \ \ \ \ \ \ +\eta(u-v-l)^{\epsilon}\partial_{t}\left((u-v-l)^{\epsilon}\right)_{+}+l\eta\partial_{t}\left((u-v-l)^{\epsilon}\right)_{+}\d z\\
= & -\int_{\Omega_{T}}((u-v-l)^{\epsilon})_{+}^{2}\partial_{t}\eta+\frac{1}{2}\eta\partial_{t}((u-v-l)^{\epsilon})_{+}^{2}\d z\\
= & -\frac{1}{2}\int_{\Omega_{T}}((u-v-l)^{\epsilon})_{+}^{2}\partial_{t}\eta\d z\\
\underset{\epsilon\rightarrow0}{\rightarrow} & -\frac{1}{2}\int_{\Omega_{T}}(u-v-l)_{+}^{2}\partial_{t}\eta\d z.
\end{align*}
Moreover, by the Lebesgue differentiation theorem for a.e.\ $s\in(0,T)$
it holds
\[
-\frac{1}{2}\int_{\Omega_{T}}(u-v-l)_{+}^{2}\partial_{t}\eta\d z=\frac{1}{4h}\int_{s-h}^{s+h}\int_{\Omega}w^{2}(x,t)\d x\d t\underset{h\rightarrow0}{\rightarrow}\frac{1}{2}\int_{\Omega}w^{2}(x,s)\d x.
\]
The terms at the right-hand side of (\ref{eq:comparison pre main est 1<p<2})
converge similarly. Hence for a.e.\ $s\in(0,T)$ we have
\begin{align}
\frac{1}{2}\int_{\Omega} & w^{2}(x,s)\d x\nonumber \\
\leq & \int_{\Omega_{s}}\left|f(Du)-f(Dv)\right|w\d z-\int_{\Omega_{s}}\left(\left|Du\right|^{p-2}Du-\left|Dv\right|^{p-2}Dv\right)\cdot Dw\d z\nonumber \\
=: & I_{1}-I_{2}.\label{eq:comparison main est 1<p<2}
\end{align}

\textbf{(Step 2)} \begingroup\allowdisplaybreaks We seek to absorb
some of $\smash{I_{1}}$ into $\smash{I_{2}}$ so that we can conclude
from Grönwall's inequality that $\smash{w\equiv0}$ almost everywhere.
Since $\smash{f}$ is locally Lipschitz continuous, there are constants
$\smash{M\geq\max(2\left\Vert Du\right\Vert _{L^{\infty}(\Omega_{T})},1)}$
and $\smash{L=L(M)}$ such that
\begin{equation}
\left|f(\xi)-f(\eta)\right|\leq L\left|\xi-\eta\right|\text{ when }\left|\xi\right|,\left|\eta\right|<M.\label{eq:comparison lipcnd}
\end{equation}
We denote $\Omega_{s}^{+}:=\left\{ x\in\Omega_{s}:w\geq0\right\} $,
\[
A:=\Omega_{s}^{\text{+}}\cap\left\{ \left|Dv\right|<M\right\} \text{ and }B:=\Omega_{s}^{+}\cap\left\{ \left|Dv\right|\geq M\right\} .
\]
Observe that in $B$ we have by the growth condition (\ref{eq:gcnd}),
choice of $M$ and the assumption that $\beta\geq1$
\begin{equation}
\left|f(Du)\right|\leq C_{f}(1+\left|Du\right|^{\beta})\leq C_{f}(M+M^{\beta})\leq2C_{f}M^{\beta}\leq2C_{f}\left|Dv\right|^{\beta}\label{eq:comparison B est}
\end{equation}
and

\begin{equation}
\left|f(Dv)\right|\leq C_{f}(1+\left|Dv\right|^{\beta})\leq2C_{f}\left|Dv\right|^{\beta}.\label{eq:comparison growthcnd}
\end{equation}
It follows from (\ref{eq:comparison lipcnd}), (\ref{eq:comparison B est}),
(\ref{eq:comparison growthcnd}) and Young's inequality that
\begin{align}
I_{1}\leq & \int_{A}L\left|Du-Dv\right|w\d z+\int_{B}\left(\left|f(Du)\right|+\left|f(Dv)\right|\right)w\d z\nonumber \\
\leq & \int_{A}L\left|Du-Dv\right|w\d z+\int_{B}4C_{f}\left|Dv\right|^{\beta}w\d z\nonumber \\
\leq & \int_{A}\epsilon\left|Du-Dv\right|^{2}+C(\epsilon,L)w^{2}\d z+\int_{B}\epsilon\left|Dv\right|^{\frac{\beta p}{\beta}}+C(\epsilon,p,\beta,L,C_{f})w^{\frac{p}{p-\beta}}\d z\nonumber \\
\leq & \epsilon\int_{A}\left|Du-Dv\right|^{2}\d z+\epsilon\int_{B}\left|Dv\right|^{p}\d z+C(\epsilon,p,\beta,L,C_{f},\left\Vert w\right\Vert _{L^{\infty}})\int_{\Omega_{s}}w^{2}\d z,\label{eq:comparison (1<p<2) I_1 est}
\end{align}
where in the last step we used that $\frac{p}{p-\beta}>2$ to estimate
\[
\int_{\Omega_{s}}w^{p/(p-\beta)}\d z=\int_{\Omega_{s}}w^{p/\left(p-\beta\right)-2}w^{2}\d z\leq\left\Vert w\right\Vert _{L^{\infty}(\Omega_{T})}^{p/(p-\beta)-2}\int_{\Omega_{s}}w^{2}\d z.
\]
Using the vector inequality
\begin{equation}
\left(\left|a\right|^{p-2}a-\left|b\right|^{p-2}b\right)\cdot\left(a-b\right)\geq\left(p-1\right)\left|a-b\right|^{2}\left(1+\left|a\right|^{2}+\left|b\right|^{2}\right)^{\frac{p-2}{2}},\label{eq:algebraic ineq 1<p<2}
\end{equation}
which holds when $1<p<2$ \cite[p98]{lindqvist_plaplace}, we get
\begin{align}
I_{2}= & \int_{\Omega_{s}}\left(\left|Du\right|^{p-2}Du-\left|Dv\right|^{p-2}Dv\right)\cdot Dw\d z\nonumber \\
\geq & (p-1)\int_{\Omega_{s}^{+}}\frac{\left|Du-Dv\right|^{2}}{\left(1+\left|Du\right|^{2}+\left|Dv\right|^{2}\right)^{\frac{2-p}{2}}}\d z\nonumber \\
\geq & (p-1)\int_{A}\frac{\left|Du-Dv\right|^{2}}{\left(1+M^{2}+M^{2}\right)^{\frac{2-p}{2}}}\d z+(p-1)\int_{B}\frac{\left(\left|Dv\right|-\left|Du\right|\right)^{2}}{\left(3\left|Dv\right|^{2}\right)^{\frac{2-p}{2}}}\d z\nonumber \\
\geq & C(p,M)\int_{A}\left|Du-Dv\right|^{2}\d z+\left(p-1\right)\int_{B}\frac{\left(\left|Dv\right|-\frac{1}{2}M\right)^{2}}{\left(3\left|Dv\right|^{2}\right)^{\frac{2-p}{2}}}\d z\nonumber \\
\geq & C(p,M)\int_{A}\left|Du-Dv\right|^{2}\d z+\left(p-1\right)\int_{B}\frac{\left(\frac{1}{2}\left|Dv\right|\right)^{2}}{\left(3\left|Dv\right|^{2}\right)^{\frac{2-p}{2}}}\d z\nonumber \\
= & C(p,M)\int_{A}\left|Du-Dv\right|^{2}\d z+C(p)\int_{B}\left|Dv\right|^{p}\d z,\label{eq:comparison (1<p<2) I_2 est}
\end{align}
where $C(p,M),C(p)>0$. Combining the estimates (\ref{eq:comparison (1<p<2) I_1 est})
and (\ref{eq:comparison (1<p<2) I_2 est}) we arrive at
\begin{align*}
I_{1}-I_{2}\leq & \left(\epsilon-C(p,M)\right)\int_{A}\left|Du-Dv\right|^{2}\d z+\left(\epsilon-C(p)\right)\int_{B}\left|Dv\right|^{p}\d z+C_{0}\int_{\Omega_{s}}w^{2}\d z,
\end{align*}
where $C_{0}=C(\epsilon,p,\beta,L,C_{f},\left\Vert w\right\Vert _{L^{\infty}})$.
Recalling (\ref{eq:comparison main est 1<p<2}) and taking small enough
$\epsilon$ yields
\[
\int_{\Omega}w^{2}(x,s)\d x\leq2C_{0}\int_{\Omega_{s}}w^{2}\d z.
\]
Since this holds for a.e.\ $s\in(0,T)$, Grönwall's inequality implies
that $w\equiv0$ a.e.\ in $\Omega_{T}$. Finally, letting $l\rightarrow0$
yields that $u-v\leq0$ a.e.\ in $\Omega_{T}$. \endgroup
\end{proof}
\begin{lem}
\label{lem:comparison lemma p>2}Let $p\geq2$ and let $f$ be locally
Lipschitz. Let $v\in L^{\infty}(\Omega_{T})$ be a weak supersolution
to (\ref{eq:p-para f}) and let $u\in L^{\infty}(\Omega_{T})$ be
a weak subsolution to 
\[
\partial_{t}u-\Delta_{p}u-f(Du)\leq-\delta\quad\text{in }\Omega_{T}
\]
for some $\delta>0$. Assume that for all $(x_{0},t_{0})\in\partial_{p}\Omega_{T}$
\[
\limsup_{(x,t)\rightarrow(x_{0},t_{0})}u(x,t)\leq\liminf_{(x,t)\rightarrow(x_{0},t_{0})}v(x,t).
\]
 Suppose also that $Du\in L^{\infty}(\Omega_{T})$. Then $u\leq v$
a.e.\ in $\Omega_{T}$.
\end{lem}
\begin{proof}
Let $l>0$ and set $w:=(u-v-l)_{+}$. Let also $s\in(0,T)$. Repeating
the first step of the proof of Lemma \ref{lem:comparison lemma 1<p<2},
we arrive at the inequality
\begin{align}
\frac{1}{2}\int_{\Omega} & w^{2}(x,s)\d x\nonumber \\
\leq & \int_{\Omega_{s}}\left|f(Du)-f(Dv)\right|w\d z-\int_{\Omega_{s}}\left(\left|Du\right|^{p-2}Du-\left|Dv\right|^{p-2}Dv\right)\cdot Dw\d z-\int_{\Omega_{s}}\delta w\d z\nonumber \\
=: & I_{1}-I_{2}-\int_{\Omega_{s}}\delta w\d z.\label{eq:comparison main est p>2}
\end{align}
Moreover, we define the constants $M$ and $L$, and the sets $A$
and $B$, exactly in the same way as in the proof of Lemma \ref{lem:comparison lemma 1<p<2}.
Then by (\ref{eq:comparison lipcnd}), (\ref{eq:comparison B est}),
(\ref{eq:comparison growthcnd}) and Young's inequality
\begin{align}
I_{1}\leq & \int_{A}L\left|Du-Dv\right|w\d z+\int_{B}\left(\left|f(Du)\right|+\left|f(Dv)\right|\right)w\d z\nonumber \\
\leq & \int_{A}\epsilon\left|Du-Dv\right|^{p}+C(\epsilon,L)w^{\frac{p}{p-1}}\d z+\int_{B}4C_{f}\left|Dv\right|^{\beta}w\d z\nonumber \\
\leq & \epsilon\int_{A}\left|Du-Dv\right|^{p}\d z+\epsilon\int_{B}\left|Dv\right|^{p}\d z+C(\epsilon,p,\beta,L,C_{f})\int_{\Omega_{s}}w^{\frac{p}{p-1}}+w^{\frac{p}{p-\beta}}\d z.\label{eq:comparison (p>2) I_1 est}
\end{align}
Using the vector inequality
\begin{equation}
\left(\left|a\right|^{p-2}a-\left|b\right|^{p-2}b\right)\cdot\left(a-b\right)\geq2^{2-p}\left|a-b\right|^{p},\label{eq:algebraic ineq p>2}
\end{equation}
which holds when $p\geq2$ \cite[p95]{lindqvist_plaplace}, we get
\begin{align*}
I_{2}\geq & C(p)\int_{A}\left|Du-Dv\right|^{p}\d z+C(p)\int_{B}\left|Du-Dv\right|^{p}\d z.
\end{align*}
Furthermore, since in $B$ it holds
\[
\left|Du-Dv\right|^{p}\geq\left(\left|Dv\right|-\left|Du\right|\right)^{p}\geq\left(\left|Dv\right|-\frac{1}{2}M\right)^{p}\geq C(p)\left|Dv\right|^{p},
\]
we arrive at
\begin{equation}
I_{2}\geq C(p)\int_{A}\left|Du-Dv\right|^{p}\d z+C(p)\int_{B}\left|Dv\right|^{p}\d z.\label{eq:comparison (p>2) I_2 est}
\end{equation}
Combining (\ref{eq:comparison (p>2) I_1 est}) and (\ref{eq:comparison (p>2) I_2 est})
with (\ref{eq:comparison main est p>2}) we get
\begin{align*}
\frac{1}{2}\int_{\Omega}w^{2}\d x\leq & \left(\epsilon-C(p)\right)\left(\int_{A}\left|Du-Dv\right|^{p}\d z+\int_{B}\left|Dv\right|^{p}\d z\right)\\
 & +\int_{\Omega_{s}}C(\epsilon,p,\beta,L,C_{f})\left(w^{\frac{p}{p-1}}+w^{\frac{p}{p-\beta}}\right)-\delta w\d z.
\end{align*}
By taking small enough $\epsilon=\epsilon(p)$, the above becomes
\begin{equation}
\int_{\Omega}w^{2}(x,s)\d x\leq\int_{\Omega_{s}}C(p,\beta,L,C_{f})\left(w^{\frac{p}{p-1}}+w^{\frac{p}{p-\beta}}\right)-\delta w\d z.\label{eq:comparison (p>2) some ineq}
\end{equation}
Observe that since $w$ is bounded and $\frac{p}{p-1},\frac{p}{p-\beta}>1$,
the integrand at the right-hand side is bounded by some constant times
$w^{2}$. To argue this rigorously, we write down the following algebraic
fact.

If $a_{0},\delta,\gamma>0$ and $\alpha>1$, then there exists $C(\alpha,\gamma,\delta,a_{0})>0$
such that
\[
\gamma a^{\alpha}\leq\delta a+C(\alpha,\gamma,\delta,a_{0})a^{2}\text{ for all }a\in[0,a_{0}).
\]
To see this, let first $\alpha<2$. Then by Young's inequality
\begin{align*}
\gamma a^{\alpha}=\gamma a\cdot a^{\alpha-1}\leq & \frac{\delta}{1+a_{0}^{\frac{2}{3-\alpha}}}a^{\frac{2}{3-\alpha}}+C(\alpha,\gamma,\delta,a_{0})a^{\left(\alpha-1\right)\cdot\frac{2}{\alpha-1}}\\
\leq & \delta a+C(\alpha,\gamma,\delta,a_{0})a^{2}.
\end{align*}
If $\alpha\geq2,$ then
\[
\gamma a^{\alpha}=\gamma a^{\alpha-2}\cdot a^{2}\leq\gamma a_{0}^{\alpha-2}a^{2}.
\]

Applying the algebraic fact on (\ref{eq:comparison (p>2) some ineq})
we get
\[
\int_{\Omega}w^{2}(x,s)\d x\leq C(p,\beta,L,C_{f},\delta,\left\Vert w\right\Vert _{L^{\infty}})\int_{\Omega_{s}}w^{2}\d z.
\]
The conclusion now follows from Grönwall's inequality and letting
$l\rightarrow0$.
\end{proof}
Next we use the previous comparison results to prove the comparison
principle for general continuous $f$.
\begin{thm}
\label{thm:comparison principle}Let $1<p<\infty$. Let $\smash{u,v\in L^{\infty}(\Omega_{T})}$
respectively be weak sub- and supersolutions to (\ref{eq:p-para f})
in $\smash{\Omega_{T}}$. Assume that for all $\smash{(x_{0},t_{0})\in\partial_{p}\Omega_{T}}$
\[
\limsup_{(x,t)\rightarrow(x_{0},t_{0})}u(x,t)\leq\liminf_{(x,t)\rightarrow(x_{0},t_{0})}v(x,t).
\]
Assume also that $Du\in L^{\infty}(\Omega_{T})$. Then $u\leq v$
a.e.\ in $\Omega_{T}$.
\end{thm}
\begin{proof}
For $\delta>0$, define
\[
u_{\delta}:=u-\frac{\delta}{T-t/2}.
\]
Then for any non-negative test function $\varphi\in C_{0}^{\infty}(\Omega_{T})$
we have by integration by parts
\begin{align*}
\int_{\Omega_{T}}-u_{\delta}\partial_{t}\varphi\d z= & \int_{\Omega_{T}}-u\partial_{t}\varphi+\frac{\delta}{T-t/2}\partial_{t}\varphi\d z\\
= & \int_{\Omega_{T}}-u\partial_{t}\varphi-\varphi\frac{\delta}{2\left(T-t/2\right)^{2}}\d z\\
\leq & \int_{\Omega_{T}}-u\partial_{t}\varphi-\varphi\frac{\delta}{2T^{2}}\d z.
\end{align*}
Since $\smash{f}$ is continuous, there is a locally Lipschitz continuous
function $\smash{f_{\delta}}$ such that $\smash{\left\Vert f-f_{\delta}\right\Vert _{L^{\infty}(\mathbb{R}^{N})}\leq\frac{\delta}{4T}}$
(see e.g.\ \cite{miculescu00}). Then, since $u$ is a weak subsolution,
we have
\begin{align*}
\int_{\Omega_{T}} & -u_{\delta}\partial_{t}\varphi+\left|Du_{\delta}\right|^{p-2}Du_{\delta}\cdot D\varphi-\varphi f_{\delta}(Du_{\delta})\d z\\
\leq & \int_{\Omega_{T}}-u\partial_{t}\varphi+\left|Du\right|^{p-2}Du\cdot D\varphi-\varphi f(Du)+\varphi\left\Vert f-f_{\delta}\right\Vert _{L^{\infty}(\mathbb{\mathbb{R}^{N}})}-\varphi\frac{\delta}{2T^{2}}\d z\\
\leq & \int_{\Omega_{T}}-\frac{\delta}{4T^{2}}\varphi\d z.
\end{align*}
Hence $u_{\delta}$ is a weak subsolution to 
\[
\partial_{t}u_{\delta}-\Delta_{p}u_{\delta}-f_{\delta}(Du_{\delta})\leq-\frac{\delta}{4T^{2}}\quad\text{in }\Omega_{T}.
\]
Similarly, since $v$ is a weak supersolution, we define
\[
v_{\delta}:=v+\frac{\delta}{T-t/2}
\]
 and deduce that $v_{\delta}$ is a weak supersolution to
\[
\partial_{t}v_{\delta}-\Delta_{p}v_{\delta}-f_{\delta}(Dv_{\delta})\geq0\quad\text{in }\Omega_{T}.
\]
Now it follows from the comparison Lemmas \ref{lem:comparison lemma 1<p<2}
and \ref{lem:comparison lemma p>2} that $u_{\delta}\leq v_{\delta}$
a.e.\ in $\Omega_{T}$. Thus
\[
u\leq v+\frac{2\delta}{T-t/2}\text{\quad a.e.\ in }\Omega_{T}.
\]
Letting $\delta\rightarrow0$ finishes the proof.
\end{proof}
Now that the comparison principle is proven, we are ready to show
that weak solutions are viscosity solutions.
\begin{thm}
\label{thm:weak is visc}Let $1<p<\infty$. Let $\smash{u\in L^{\infty}(\Xi)}$
be a lower semicontinuous weak supersolution to (\ref{eq:p-para f})
in $\Xi$. Then $u$ is a viscosity supersolution to (\ref{eq:p-para f})
in $\Xi$.
\end{thm}
\begin{proof}
Assume on the contrary that there is $\smash{\phi\in C^{2}(\Xi)}$
touching $u$ from below at $\smash{(x_{0},t_{0})\in\Xi}$, $\smash{D\phi(x,t)\not=0}$
for $\smash{x\not=x_{0}}$ and
\begin{equation}
\limsup_{\substack{\substack{(x,t)\rightarrow(x_{0},t_{0})\\
x\not=x_{0}
}
}
}\left(\partial_{t}\phi(x,t)-\Delta_{p}\phi(x,t)-f(D\phi(x,t))\right)<0.\label{eq:weak is visc lipschitz 1<p<2}
\end{equation}
Denote $Q_{r}:=B_{r}(x_{0})\times\left(t_{0}-r,t_{0}+r\right)$. It
follows from above that there are $r>0$ and $\delta>0$ such that
\begin{equation}
\partial_{t}\phi-\Delta_{p}\phi-f(D\phi)<-\delta\quad\text{in }Q_{r}\setminus\left\{ x=x_{0}\right\} .\label{eq:weak is visc lipschitz 1<p<2 someineq}
\end{equation}
Indeed, otherwise there would be a sequence $(x_{n},t_{n})\rightarrow(x_{0},t_{0})$
such that $x_{n}\not=x_{0}$ and 
\[
\partial_{t}\phi(x_{n},t_{n})-\Delta_{p}\phi(x_{n},t_{n})-f(D\phi(x_{n},t_{n}))>-\frac{1}{n},
\]
but this contradicts (\ref{eq:weak is visc lipschitz 1<p<2}). Using
Gauss's theorem and (\ref{eq:weak is visc lipschitz 1<p<2 someineq})
we obtain for any non-negative test function $\varphi\in C_{0}^{\infty}(Q_{r})$
that 
\begin{align*}
\int_{Q_{r}} & -\phi\partial_{t}\varphi+\left|D\phi\right|^{p-2}D\phi\cdot D\varphi-\varphi f(D\phi)\d z\\
= & \lim_{\rho\rightarrow0}\int_{Q_{r}\setminus\left\{ \left|x-x_{0}\right|\leq\rho\right\} }-\phi\partial_{t}\varphi+\left|D\phi\right|^{p-2}D\phi\cdot D\varphi-\varphi f(D\phi)\d z\\
= & \lim_{\rho\rightarrow0}\big(\int_{Q_{r}\setminus\left\{ \left|x-x_{0}\right|\leq\rho\right\} }\varphi\partial_{t}\phi-\varphi\div(\left|D\phi\right|^{p-2}D\phi)-\varphi f(D\phi)\d z\\
 & \ \ \ \ \ \ +\int_{t_{0}-r}^{t_{0}+r}\int_{\left\{ \left|x-x_{0}\right|=\rho\right\} }\varphi\left|D\phi\right|^{p-2}D\phi\cdot\frac{(x-x_{0})}{\rho}\d S\d t\big)\\
= & \lim_{\rho\rightarrow0}\int_{Q_{r}\setminus\left\{ \left|x-x_{0}\right|\leq\rho\right\} }\varphi\left(\partial_{t}\phi-\Delta_{p}\phi-f(D\phi)\right)\d z\\
\leq & \int_{Q_{r}}-\delta\varphi\d z.
\end{align*}
Let $l:=\min_{\partial_{p}Q_{r}}\left(u-\phi\right)>0$ and set $\widetilde{\phi}:=\phi+l$.
Then by the above inequality, $\widetilde{\phi}$ is a weak subsolution
to
\[
\partial_{t}\widetilde{\phi}-\Delta_{p}\widetilde{\phi}-f(D\widetilde{\phi})\leq-\delta\quad\text{in }Q_{r}
\]
and on $\partial_{p}Q_{r}$ it holds $\tilde{\phi}=\phi+l\leq\phi+u-\phi=u$.
Hence Theorem \ref{thm:comparison principle} implies that $\widetilde{\phi}\leq u$
in $Q_{r}$. But this is not possible since $\widetilde{\phi}(x_{0},t_{0})>u(x_{0},t_{0})$.\begin{old}\textbf{(Case
$p\geq2$)} Assume on contrary that there is a $\phi\in C^{2}(\overline{\Omega}\times[0,T])$
testing $u$ from below at $(x_{0},t_{0})\in\Omega_{T}$ and
\[
\partial_{t}\phi(x_{0},t_{0})-\Delta_{p}\phi(x_{0},t_{0})-f(D\phi(x_{0},t_{0}))<-\delta
\]
for some $\delta>0$. By continuity we have $r>0$ such that 
\[
\partial_{t}\phi-\Delta_{p}\phi-f(D\phi)<-\delta\text{ in }Q_{r}(x_{0},t_{0}).
\]
Multiplying the above inequality by $\varphi\in C_{0}^{\infty}(Q_{r}(x_{0},t_{0}))$
and integrating by parts we arrive at
\[
\int_{\Omega_{T}}\phi\partial_{t}\varphi+\left|D\phi\right|^{p-2}D\phi\cdot D\varphi-\varphi f(D\phi)\d z\leq-\int_{\Omega_{T}}\delta\varphi\d z.
\]
Let $l=\min_{\partial_{p}Q_{r}(x_{0},t_{0})}u-\phi$ and set $\tilde{\phi}=\phi+l$.
Then $\tilde{\phi}$ is a weak subsolution to 
\[
\partial_{t}\tilde{\phi}-\Delta_{p}\tilde{\phi}-f(D\tilde{\phi})\text{ in }Q_{r}(x_{0},t_{0})
\]
and on $\partial_{p}Q_{r}(x_{0},t_{0})$ it holds that $\tilde{\phi}=\phi+l\leq\phi+u-\phi=u$.
Hence Lemma \ref{lem:comparison lemma p>2} implies that $\tilde{\phi}\leq u$
in $Q_{r}(x_{0},t_{0})$. But this is not possible since $\tilde{\phi}(x_{0},t_{0})>u(x_{0},t_{0})$.\end{old}
\end{proof}

\section{Viscosity solutions are weak solutions}

We show that bounded viscosity supersolutions to (\ref{eq:p-para f})
are weak supersolutions when $1<p<\infty$ and $\smash{f\in C(\mathbb{R}^{N})}$
satisfies the growth condition (\ref{eq:gcnd}). We use the method
developed in \cite{newequivalence}. The method of \cite{newequivalence}
was previously applied to parabolic equations in \cite{ParvVaz},
but for radially symmetric solutions. 

The idea is to approximate a viscosity supersolution $u$ to (\ref{eq:p-para f})
by the \textit{inf-convolution}
\[
u_{\varepsilon}(x,t):=\inf_{(y,s)\in\Xi}\left\{ u(y,s)+\frac{\left|x-y\right|^{q}}{q\varepsilon^{q-1}}+\frac{\left|t-s\right|^{2}}{2\varepsilon}\right\} ,
\]
where $\varepsilon>0$ and $q\geq2$ is a fixed constant so large
that $\smash{p-2+\frac{q-2}{q-1}>0}$. It is straightforward to show
that the inf-convolution $\smash{u_{\varepsilon}}$ is a viscosity
supersolution in the smaller domain 
\[
\Xi_{\varepsilon}=\left\{ (x,t)\in\Xi:B_{r(\varepsilon)}(x)\times(t-t(\varepsilon),t+t(\varepsilon))\Subset\Xi\right\} ,
\]
where $r(\varepsilon),$ $t(\varepsilon)\rightarrow0$ as $\varepsilon\rightarrow0$.
Moreover, $\smash{u_{\varepsilon}}$ is semi-concave by definition
and therefore it has a second derivative almost everywhere. It follows
from these pointwise properties that $\smash{u_{\varepsilon}}$ is
a weak supersolution to (\ref{eq:p-para f}) in $\smash{\Xi_{\varepsilon}}$.
Caccioppoli type estimates then imply that $\smash{u_{\varepsilon}}$
converges to $u$ in a parabolic Sobolev space and consequently $u$
is a weak supersolution.

The standard properties of the inf-convolution are postponed to the
end of this section. Instead, we begin by proving the key observation:
that the inf-convolution of a viscosity supersolution is a weak supersolution
in the smaller domain $\smash{\Xi_{\varepsilon}}$. When $p\geq2$,
the idea is the following. Since $\smash{u_{\varepsilon}}$ is a viscosity
supersolution to (\ref{eq:p-para f}) that is twice differentiable
almost everywhere, it satisfies the equation pointwise almost everywhere.
Hence we may multiply the equation by a non-negative test function
$\varphi$ and integrate over $\smash{\Xi_{\varepsilon}}$ so that
the integral will be non-negative. Then we approximate this expression
through smooth functions $\smash{u_{\varepsilon,j}}$ defined via
the standard mollification. Since $\smash{u_{\varepsilon,j}}$ is
smooth, we may integrate by parts to reach the weak formulation of
the equation, see (\ref{eq:inf_conv is weak mid eq}). It then remains
to let $j\rightarrow\infty$ to conclude that $\smash{u_{\varepsilon}}$
is a weak supersolution. The range $1<p<2$ is more delicate because
of the singularity of the $p$-Laplace operator
\[
\Delta_{p}u:=\left|Du\right|^{p-2}\left(\Delta u+\frac{(p-2)}{\left|Du\right|^{2}}\left\langle D^{2}uDu,Du\right\rangle \right),
\]
and therefore we consider the case $p\geq2$ first.
\begin{lem}
Let $p\geq2$. Let $u$ be a bounded viscosity supersolution to (\ref{eq:p-para f})
in $\Xi$. Then $u_{\varepsilon}$ is a weak supersolution to (\ref{eq:p-para f})
in $\Xi_{\varepsilon}$.
\end{lem}
\begin{proof}
Fix a non-negative test function $\varphi\in C_{0}^{\infty}(\Xi_{\varepsilon})$.
By Remark \ref{rem:inf_remark}, the function
\[
\phi(x,t):=u_{\varepsilon}(x,t)-C(q,\varepsilon,u)\left(\left|x\right|^{2}+t^{2}\right)
\]
is concave in $\smash{\Xi_{\varepsilon}}$ and we can approximate
it by smooth concave functions $\smash{\phi_{j}}$ so that $\smash{\left(\phi_{j},\partial_{t}\phi_{j},D\phi_{j},D^{2}\phi_{j}\right)\rightarrow\left(\phi,\partial_{t}\phi,D\phi,D^{2}\phi\right)}$
a.e.$\:$in $\smash{\Xi_{\varepsilon}}$. We define
\[
u_{\varepsilon,j}(x,t):=\phi_{j}(x,t)+C(q,\varepsilon,u)\left(\left|x\right|^{2}+t^{2}\right).
\]
Since $u_{\varepsilon,j}$ is smooth and $\varphi$ is compactly supported
in $\Xi_{\varepsilon}$, we integrate by parts to get
\begin{align}
\int_{\Xi_{\varepsilon}} & \varphi\bigg(\partial_{t}u_{\varepsilon,j}-\left|Du_{\varepsilon,j}\right|^{p-2}\bigg(\Delta u_{\varepsilon,j}+\frac{(p-2)}{\left|Du_{\varepsilon,j}\right|^{2}}\left\langle D^{2}u_{\varepsilon,j}Du_{\varepsilon,j},Du_{\varepsilon,j}\right\rangle \bigg)-f(Du_{\varepsilon,j})\bigg)\d z\nonumber \\
= & \int_{\Xi_{\varepsilon}}\varphi\partial_{t}u_{\varepsilon,j}-\varphi\div\left(\left|Du_{\varepsilon,j}\right|^{p-2}Du_{\varepsilon,j}\right)-\varphi f(Du_{\varepsilon,j})\d z\nonumber \\
= & \int_{\Xi_{\varepsilon}}-u_{\varepsilon,j}\partial_{t}\varphi+\left|Du_{\varepsilon,j}\right|^{p-2}Du_{\varepsilon,j}\cdot D\varphi-\varphi f(Du_{\varepsilon,j})\d z.\label{eq:inf_conv is weak mid eq}
\end{align}
This implies that
\begin{align*}
\liminf_{j\rightarrow\infty}\!\int_{\Xi_{\varepsilon}} & \varphi\bigg(\partial_{t}u_{\varepsilon,j}-\left|Du_{\varepsilon,j}\right|^{p-2}\bigg(\Delta u_{\varepsilon,j}+\frac{(p-2)}{\left|Du_{\varepsilon,j}\right|^{2}}\left\langle D^{2}u_{\varepsilon,j}Du_{\varepsilon,j},Du_{\varepsilon,j}\right\rangle \bigg)\!-f(Du_{\varepsilon,j})\bigg)\d z\\
\leq & \lim_{j\rightarrow\infty}\int_{\Xi_{\varepsilon}}-u_{\varepsilon,j}\partial_{t}\varphi+\left|Du_{\varepsilon,j}\right|^{p-2}Du_{\varepsilon,j}\cdot D\varphi-\varphi f(Du_{\varepsilon,j})\d z.
\end{align*}
 We intend to use Fatou's lemma at the left-hand side and dominated
convergence at the right-hand side. Once we verify their assumptions,
we arrive at the inequality
\[
\int_{\Xi_{\varepsilon}}\varphi\left(\partial_{t}u_{\varepsilon}-\Delta_{p}u_{\varepsilon}-f(Du_{\varepsilon})\right)\d z\leq\int_{\Xi_{\varepsilon}}-u_{\varepsilon}\partial_{t}\varphi+\left|Du_{\varepsilon}\right|^{p-2}Du_{\varepsilon}\cdot D\varphi-\varphi f(Du_{\varepsilon})\d z.
\]
The left-hand side is non-negative since by Lemma \ref{lem:infconv is viscsup}
the inf-convolution $u_{\varepsilon}$ is still a viscosity supersolution
in $\Xi_{\varepsilon}$. Consequently $\smash{u_{\varepsilon}}$ is
a weak supersolution in $\Xi_{\varepsilon}$ as desired. It remains
to justify our use of Fatou's lemma and the dominated convergence
theorem. It follows from Remark \ref{rem:inf_remark} that $\smash{\left|u_{\varepsilon,j}\right|}$,
$\smash{\left|\partial_{t}u_{\varepsilon,j}\right|}$ and $\smash{\left|Du_{\varepsilon,j}\right|}$
are uniformly bounded by some constant $\smash{M>0}$ in the support
of $\varphi$ with respect to $j$. This justifies our use of the
dominated convergence theorem. Observe then that since $\phi_{j}$
is concave, we have $\smash{D^{2}u_{\varepsilon,j}\leq C(q,\varepsilon,u)I}$.
Hence
\begin{align*}
\partial_{t}u_{\varepsilon,j} & -\left|Du_{\varepsilon,j}\right|^{p-2}\bigg(\Delta u_{\varepsilon,j}+\frac{(p-2)}{\left|Du_{\varepsilon,j}\right|^{2}}\left\langle D^{2}u_{\varepsilon,j}Du_{\varepsilon,j},Du_{\varepsilon,j}\right\rangle \bigg)-f(Du_{\varepsilon,j})\\
\geq & -M-C(q,\varepsilon,u)M^{p-2}\left(N+p-2\right)-\sup_{\left|\xi\right|\leq M}\left|f(\xi)\right|.
\end{align*}
The integrand at the left-hand side of (\ref{eq:inf_conv is weak mid eq})
is therefore bounded from below with respect to $j$, justifying our
use of Fatou's lemma.
\end{proof}
Next we consider the singular case $\smash{1<p<2}$. We cannot directly
repeat the previous proof because $\smash{\Delta_{p}u_{\varepsilon}}$
no longer has a clear meaning at the points where $\smash{Du_{\varepsilon}=0}$.
To deal with this, we consider the regularized terms
\begin{equation}
\Delta_{p,\delta}u:=\left(\delta+\left|Du\right|^{2}\right)^{\frac{p-2}{2}}\bigg(\Delta u+\frac{p-2}{\delta+\left|Du\right|^{2}}\Delta_{\infty}u\bigg),\label{eq:regplap}
\end{equation}
where $\Delta_{\infty}u=\left\langle D^{2}uDu,Du\right\rangle .$
\begin{lem}
\label{lem:inf conv is weak}Let $1<p<2$ . Let $u$ be a bounded
viscosity supersolution to (\ref{eq:p-para f}) in $\Xi$. Then $u_{\varepsilon}$
is a weak supersolution to (\ref{eq:p-para f}) in $\Xi_{\varepsilon}$.
\end{lem}
\begin{proof}
\textbf{(Step 1) }Let $\varphi\in C_{0}^{\infty}(\Xi_{\varepsilon})$
be a non-negative test function. We set 
\begin{equation}
\phi(x,t):=u_{\varepsilon}(x,t)-C(q,\varepsilon,u)\left(\left|x\right|^{2}+t^{2}\right),\label{eq:inf conv is weak concavity p<2}
\end{equation}
where $\smash{C(q,\varepsilon,u)}$ is the semi-concavity constant
of $\smash{u_{\varepsilon}}$ in $\smash{\Xi_{\varepsilon}}$. Then
by Remark \ref{rem:inf_remark} we can approximate $\phi$ by smooth
concave functions $\smash{\phi_{j}}$ so that $\smash{\left(\phi_{j},\partial_{t}\phi_{j},D\phi_{j},D^{2}\phi_{j}\right)}\rightarrow\smash{\left(\phi,\partial_{t}\phi,D\phi,D^{2}\phi\right)}$
a.e.\ in $\Xi_{\varepsilon}$. We define
\[
u_{\varepsilon,j}(x,t):=\phi_{j}(x,t)+C(q,\varepsilon,u)\left(\left|x\right|^{2}+t^{2}\right).
\]
 Let $\delta\in(0,1)$. Since $u_{\varepsilon,j}$ is smooth and $\varphi$
is compactly supported in $\Xi_{\varepsilon}$, we calculate via integration
by parts
\begin{align*}
\int_{\Xi_{\varepsilon}} & \varphi\bigg(\partial_{t}u_{\varepsilon,j}-\left(\delta+\left|Du_{\varepsilon,j}\right|^{2}\right)^{\frac{p-2}{2}}\bigg(\Delta u_{\varepsilon,j}+\frac{p-2}{\delta+\left|Du_{\varepsilon,j}\right|^{2}}\Delta_{\infty}u_{\varepsilon,j}\bigg)-f(Du_{\varepsilon,j})\bigg)\d z\\
= & \int_{\Xi_{\varepsilon}}\varphi\partial_{t}u_{\varepsilon,j}-\varphi\div\left(\left(\delta+\left|Du_{\varepsilon,j}\right|^{2}\right)^{\frac{p-2}{2}}Du_{\varepsilon,j}\right)-\varphi f(Du_{\varepsilon,j})\d z\\
= & \int_{\Xi_{\varepsilon}}-u_{\varepsilon,j}\partial_{t}\varphi+\left(\delta+\left|Du_{\varepsilon,j}\right|^{2}\right)^{\frac{p-2}{2}}Du_{\varepsilon,j}\cdot D\varphi-\varphi f(Du_{\varepsilon,j})\d z.
\end{align*}
Recalling the shorthand $\Delta_{p,\delta}$ defined in (\ref{eq:regplap}),
we deduce from the above that
\begin{align}
\liminf_{j\rightarrow\infty} & \int_{\Xi_{\varepsilon}}\varphi\left(\partial_{t}u_{\varepsilon,j}-\Delta_{p,\delta}u_{\varepsilon,j}-f(Du_{\varepsilon,j})\right)\d z\nonumber \\
\leq & \lim_{j\rightarrow\infty}\int_{\Xi_{\varepsilon}}-u_{\varepsilon,j}\partial_{t}\varphi+\left(\delta+\left|Du_{\varepsilon,j}\right|^{2}\right)^{\frac{p-2}{2}}Du_{\varepsilon,j}\cdot D\varphi-\varphi f(Du_{\varepsilon,j})\d z.\label{eq:inf conv is weak p<2}
\end{align}
We use Fatou's lemma at the left-hand side and the dominated convergence
at the right-hand side. Once we verify their assumptions, we arrive
at the auxiliary inequality
\begin{align}
\int_{\Xi_{\varepsilon}} & \varphi\left(\partial_{t}u_{\varepsilon}-\Delta_{p,\delta}u_{\varepsilon}-f(Du_{\varepsilon})\right)\d z\nonumber \\
\leq & \int_{\Xi_{\varepsilon}}-u_{\varepsilon}\partial_{t}\varphi+\left(\delta+\left|Du_{\varepsilon}\right|^{2}\right)^{\frac{p-2}{2}}Du_{\varepsilon}\cdot D\varphi-\varphi f(Du_{\varepsilon})\d z.\label{eq:inf conv is weak aux p<2}
\end{align}
Next we verify the assumptions of Fatou's lemma and the dominated
convergence theorem. By Remark \ref{rem:inf_remark}, the functions
$\smash{\left|u_{\varepsilon,j}\right|}$, $\smash{\left|\partial_{t}u_{\varepsilon,j}\right|}$
and $\smash{\left|Du_{\varepsilon,j}\right|}$ are uniformly bounded
by some constant $M>1$ in the support of $\varphi$ with respect
to $j$. Hence the assumptions of the dominated convergence theorem
are satisfied. Observe then that the concavity of $\smash{\phi_{j}}$
implies that $\smash{D^{2}u_{\varepsilon,j}\leq C(q,\varepsilon,u)I}$.
Thus the integrand at the left-hand side of (\ref{eq:inf conv is weak p<2})
has a lower bound independent of $j$ when $\smash{Du_{\varepsilon,j}=0}$.
When $\smash{Du_{\varepsilon,j}\not=0}$, we have\begingroup\allowdisplaybreaks
\begin{align*}
\partial_{t} & u_{\varepsilon,j}-\left(\delta+\left|Du_{\varepsilon,j}\right|^{2}\right)^{\frac{p-2}{2}}\bigg(\Delta u_{\varepsilon,j}+\frac{p-2}{\delta+\left|Du_{\varepsilon,j}\right|^{2}}\Delta_{\infty}u_{\varepsilon,j}\bigg)-f(Du_{\varepsilon,j})\\
= & \partial_{t}u_{\varepsilon,j}-\frac{\left(\delta+\left|Du_{\varepsilon,j}\right|^{2}\right)^{\frac{p-2}{2}}}{\delta+\left|Du_{\varepsilon,j}\right|^{2}}\bigg(\left|Du_{\varepsilon,j}\right|^{2}\bigg(\Delta u_{\varepsilon,j}+\frac{p-2}{\left|Du_{\varepsilon,j}\right|^{2}}\Delta_{\infty}u_{\varepsilon,j}\bigg)+\delta\Delta u_{\varepsilon,j}\bigg)-f(Du_{\varepsilon,j})\\
\geq & \partial_{t}u_{\varepsilon,j}-\frac{\left(\delta+\left|Du_{\varepsilon,j}\right|^{2}\right)^{\frac{p-2}{2}}}{\delta+\left|Du_{\varepsilon,j}\right|^{2}}C(q,\varepsilon,u)\left(\left|Du_{\varepsilon,j}\right|^{2}\left(N+p-2\right)+\delta N\right)-f(Du_{\varepsilon,j})\\
\geq & \partial_{t}u_{\varepsilon,j}-C(q,\varepsilon,u)\left(\delta+\left|Du_{\varepsilon,j}\right|^{2}\right)^{\frac{p-2}{2}}\left(2N+p-2\right)-f(Du_{\varepsilon,j})\\
\geq & -M-C(q,\varepsilon,u)\delta^{\frac{p-2}{2}}\left(2N+p-2\right)-\sup_{\left|\xi\right|\leq M}\left|f(\xi)\right|,
\end{align*}
so that our use of Fatou's lemma is justified.

\textbf{(Step 2)} We let $\delta\rightarrow0$ in the auxiliary inequality
(\ref{eq:inf conv is weak aux p<2}). Since $u_{\varepsilon}$ is
Lipschitz continuous, the dominated convergence theorem implies 
\begin{align}
\liminf_{\delta\rightarrow0} & \int_{\Xi_{\varepsilon}}\varphi\left(\partial_{t}u_{\varepsilon}-\Delta_{p,\delta}u_{\varepsilon}-f(Du_{\varepsilon})\right)\d z\nonumber \\
\leq & \int_{\Xi_{\varepsilon}}-u_{\varepsilon}\partial_{t}\varphi+\left|Du_{\varepsilon}\right|^{p-2}Du_{\varepsilon}\cdot D\varphi-\varphi f(Du_{\varepsilon})\d z.\label{eq:xiao1}
\end{align}
Applying Fatou's lemma (we verify assumptions at the end), we get
\begin{align}
\liminf_{\delta\rightarrow0} & \int_{\Xi_{\varepsilon}}\varphi\left(\partial_{t}u_{\varepsilon}-\Delta_{p,\delta}u_{\varepsilon}-f(Du_{\varepsilon})\right)\d z\nonumber \\
\geq & \int_{\Xi_{\varepsilon}}\liminf_{\delta\rightarrow0}\varphi\left(\partial_{t}u_{\varepsilon}-\Delta_{p,\delta}u_{\varepsilon}-f(Du_{\varepsilon})\right)\d z\nonumber \\
= & \int_{\Xi_{\varepsilon}\cap\left\{ Du_{\varepsilon}\not=0\right\} }\liminf_{\delta\rightarrow0}\varphi\left(\partial_{t}u_{\varepsilon}-\Delta_{p,\delta}u_{\varepsilon}-f(Du_{\varepsilon})\right)\d z\nonumber \\
 & +\int_{\Xi_{\varepsilon}\cap\left\{ Du_{\varepsilon}=0\right\} }\liminf_{\delta\rightarrow0}\varphi(\partial_{t}u_{\varepsilon}-\delta^{\frac{p-2}{2}}\Delta u_{\varepsilon}-f(0))\d z\nonumber \\
= & \int_{\Xi_{\varepsilon}\cap\left\{ Du_{\varepsilon}\not=0\right\} }\varphi\left(\partial_{t}u_{\varepsilon}-\Delta_{p}u_{\varepsilon}-f(Du_{\varepsilon})\right)\d z\nonumber \\
 & +\int_{\Xi_{\varepsilon}\cap\left\{ Du_{\varepsilon}=0\right\} }\varphi\left(\partial_{t}u_{\varepsilon}-f(0)\right)\d z\geq0,\label{eq:xiao2}
\end{align}
where the last inequality follows from Lemma \ref{lem:infconv is viscsup}
since $u_{\varepsilon}$ is twice differentiable almost everywhere.
Combining (\ref{eq:xiao1}) and (\ref{eq:xiao2}), we find that $u_{\varepsilon}$
is a weak supersolution in $\Xi_{\varepsilon}$. It remains to verify
the assumptions of Fatou's lemma, i.e.\ that the integrand at the
left-hand side of (\ref{eq:xiao1}) has a lower bound independent
of $\delta$. When $Du_{\varepsilon}=0$, this follows directly from
the inequality 
\[
D^{2}u_{\varepsilon}\leq\frac{q-1}{\varepsilon}\left|Du_{\varepsilon}\right|^{\frac{q-1}{q-2}}I,
\]
which holds by Lemma \ref{lem:inf conv prop}. When $Du_{\varepsilon}\not=0$,
we recall that by Lipschitz continuity $\partial_{t}u_{\varepsilon}$
and $Du_{\varepsilon}$ are uniformly bounded in $\Xi_{\varepsilon}$,
and estimate
\begin{align*}
 & -\left(\delta+\left|Du_{\varepsilon}\right|^{2}\right)^{\frac{p-2}{2}}\bigg(\Delta u_{\varepsilon}+\frac{p-2}{\delta+\left|Du_{\varepsilon}\right|^{2}}\Delta_{\infty}u_{\varepsilon}\bigg)\\
 & \ =-\frac{\left(\delta+\left|Du_{\varepsilon}\right|^{2}\right)^{\frac{p-2}{2}}}{\delta+\left|Du_{\varepsilon}\right|^{2}}\left(\left|Du_{\varepsilon}\right|^{2}\left(\Delta u_{\varepsilon}+\frac{p-2}{\left|Du_{\varepsilon}\right|^{2}}\Delta_{\infty}u_{\varepsilon}\right)+\delta\Delta u_{\varepsilon}\right)\\
 & \ \geq-\frac{\left(\delta+\left|Du_{\varepsilon}\right|^{2}\right)^{\frac{p-2}{2}}}{\delta+\left|Du_{\varepsilon}\right|^{2}}\frac{\left(q-1\right)}{\varepsilon}\left(\left|Du_{\varepsilon}\right|^{\frac{q-2}{q-1}+2}\left(N+p-2\right)+\left|Du_{\varepsilon}\right|^{\frac{q-2}{q-1}}\delta N\right)\\
 & \ \geq-\left(\delta+\left|Du_{\varepsilon}\right|^{2}\right)^{\frac{p-2}{2}}\frac{\left(q-1\right)}{\varepsilon}\left|Du_{\varepsilon}\right|^{\frac{q-2}{q-1}}\left(2N+p-2\right)\\
 & \ \geq-\left|Du_{\varepsilon}\right|^{p-2+\frac{q-2}{q-1}}\frac{\left(q-1\right)}{\varepsilon}\left(2N+p-2\right)\\
 & \ \geq-\left\Vert Du_{\varepsilon}\right\Vert _{L^{\infty}(\Xi_{\varepsilon})}^{p-2+\frac{q-2}{q-1}}\frac{\left(q-1\right)}{\varepsilon}\left(2N+p-2\right),
\end{align*}
where we used that $p-2+\frac{q-2}{q-1}>0$. Hence the assumptions
of Fatou's lemma hold.\endgroup
\end{proof}
If $\smash{u_{\varepsilon}}$ is the sequence of inf-convolutions
of a viscosity supersolution to (\ref{eq:p-para f}), then by next
Caccioppoli's inequality the sequence $\smash{Du_{\varepsilon}}$
converges weakly in $\smash{L_{loc}^{p}(\Xi)}$ up to a subsequence.
However, we need stronger convergence to pass to the limit under the
integral sign of
\[
\int_{\Xi}-\varphi\partial_{t}u_{\varepsilon}+\left|Du_{\varepsilon}\right|^{p-2}Du_{\varepsilon}\cdot D\varphi-\varphi f(Du_{\varepsilon})\d z\geq0.
\]
For this end, we show in Lemma \ref{lem:lr conv} that $\smash{Du_{\varepsilon}}$
converges in $\smash{L_{loc}^{r}(\Xi)}$ for all $1<r<p$. 
\begin{lem}[Caccioppoli's inequality]
\label{lem:Caccioppoli} Let $1<p<\infty$. Assume that $u$ is a
locally Lipschitz continuous weak supersolution to (\ref{eq:p-para f})
in $\Xi$. Then there is a constant $C=C(p,\beta,C_{f})$ such that
for any test function $\xi\in C_{0}^{\infty}(\Xi)$ we have
\[
\int_{\Xi}\xi^{p}\left|Du\right|^{p}\d z\leq C\int_{\Xi}M^{2}\partial_{t}\xi^{p}+M^{p}\left|D\xi\right|^{p}+(M^{\frac{p}{p-\beta}}+M)\xi^{p}\d z,
\]
where $M=\left\Vert u\right\Vert _{L^{\infty}(\supp\xi)}$.
\end{lem}
\begin{proof}
Since $u$ is locally Lipschitz continuous, the function $\varphi:=\left(M-u\right)\xi^{p}$
is an admissible test function. Testing the weak formulation of (\ref{eq:p-para f})
with $\varphi$ yields
\begin{align}
\int_{\Xi}\xi^{p}\left|Du\right|^{p}\d z\leq & \int_{\Xi}u\partial_{t}\varphi+p\xi^{p-1}(M-u)\left|Du\right|^{p-1}\left|D\xi\right|+\varphi f(Du)\d z.\label{eq:caccioppoli 1}
\end{align}
We have by integration by parts
\begin{align*}
\int_{\Xi}u\partial_{t}\varphi\d z= & \int_{\Xi}-\xi^{p}u\partial_{t}u+u(M-u)\partial_{t}\xi^{p}\d z\\
= & \int_{\Xi}-\frac{1}{2}\xi^{p}\partial_{t}u^{2}+u(M-u)\partial_{t}\xi^{p}\d z\\
= & \int_{\Xi}\frac{1}{2}u^{2}\partial_{t}\xi^{p}+u(M-u)\partial_{t}\xi^{p}\d z\leq\int_{\Xi}CM^{2}\partial_{t}\xi^{p}\d z.
\end{align*}
By Young's inequality

\begin{align*}
\int_{\Xi} & p\xi^{p-1}(M-u)\left|Du\right|^{p-1}\left|D\xi\right|\d z\leq\int_{\Xi}\frac{1}{4}\xi^{p}\left|Du\right|^{p}\d z+C(p)\int_{\Xi}M^{p}\left|D\xi\right|^{p}\d z.
\end{align*}
Using the growth condition (\ref{eq:gcnd}) and Young's inequality
we get
\begin{align*}
\int_{\Xi}\varphi f(Du)\d z\leq & \int_{\Xi}\left(M-u\right)\xi^{p}C_{f}\left(1+\left|Du\right|^{\beta}\right)\d z\\
= & \int_{\Xi}C_{f}\left(M-u\right)\xi^{p-\beta}\xi^{\beta}\left|Du\right|^{\beta}+C_{f}(M-u)\xi^{p}\d z\\
\leq & \int_{\Xi}\frac{1}{4}\xi^{p}\left|Du\right|^{p}+C(p,\beta,C_{f})\left(M-u\right)^{\frac{p}{p-\beta}}\xi^{p}+C_{f}\left(M-u\right)\xi^{p}\d z\\
\leq & \int_{\Xi}\frac{1}{4}\xi^{p}\left|Du\right|^{p}+C(p,\beta,C_{f})\left(M^{\frac{p}{p-\beta}}+M\right)\xi^{p}\d z.
\end{align*}
Combining these estimates with (\ref{eq:caccioppoli 1}) and absorbing
the terms with $Du$ to the left-hand side yields the desired inequality.
\end{proof}
The proof of Lemma \ref{lem:lr conv} is based on that of Lemma 5
in \cite{lindqvistManfredi07}, see also Theorem 5.3 in \cite{korteKuusiParv10}.
For the convenience of the reader, we give the full details.
\begin{lem}
\label{lem:lr conv} Let $\smash{1<p<\infty}$. Suppose that $\smash{\left(u_{j}\right)}$
is a sequence of locally Lipschitz continuous weak supersolutions
to (\ref{eq:p-para f}) such that $\smash{u_{j}\rightarrow u}$ locally
uniformly in $\Xi$. Then $\smash{\left(Du_{j}\right)}$ is a Cauchy
sequence in $\smash{L_{loc}^{r}(\Xi)}$ for any $\smash{1<r<p}$.
\end{lem}
\begin{proof}
Let $U\Subset\Xi$ and take a cut-off function $\theta\in C_{0}^{\infty}(\Xi)$
such that $0\leq\theta\leq1$ and $\theta\equiv1$ in $U$. For $\delta>0$,
we set
\[
w_{jk}=\begin{cases}
\delta, & u_{j}-u_{k}>\delta,\\
u_{j}-u_{k}, & \left|u_{j}-u_{k}\right|\leq\delta,\\
-\delta, & u_{j}-u_{k}<-\delta.
\end{cases}
\]
Then the function $(\delta-w_{jk})\theta$ is an admissible test function
with a time derivative since it is Lipschitz continuous. Since $u_{j}$
is a weak supersolution, testing the weak formulation of (\ref{eq:p-para f})
with $(\delta-w_{jk})\theta$ yields
\begin{align*}
0\leq & \int_{\Xi}-u_{j}\partial_{t}((\delta-w_{jk})\theta)+\left|Du_{j}\right|^{p-2}Du_{j}\cdot D((\delta-w_{jk})\theta)-(\delta-w_{jk})\theta f(Du_{j})\d z\\
= & \int_{\Xi}-\theta\left|Du_{j}\right|^{p-2}Du_{j}\cdot Dw_{jk}+(\delta-w_{jk})\left|Du_{j}\right|^{p-2}Du_{j}\cdot D\theta-(\delta-w_{jk})\theta f(Du_{j})\\
 & \ \ \ \ +u_{j}\partial_{t}(w_{jk}\theta)-(\delta-w_{jk})u_{j}\partial_{t}\theta\d z.
\end{align*}
Since $\left|w_{jk}\right|\leq\delta$ and $Dw_{jk}=\chi_{\left\{ \left|u_{j}-u_{k}\right|<\delta\right\} }\left(Du_{j}-Du_{k}\right)$,
the above becomes
\begin{align*}
\int_{\left\{ \left|u_{j}-u_{k}\right|<\delta\right\} } & \theta\left|Du_{j}\right|^{p-2}Du_{j}\cdot\left(Du_{j}-Du_{k}\right)\d z\\
\leq & \int_{\Xi}2\delta\left|Du_{j}\right|^{p-1}\left|D\theta\right|+2\delta\theta\left|f(Du_{j})\right|+u_{j}\partial_{t}(w_{jk}\theta)+2\delta\left|u_{j}\right|\left|\partial_{t}\theta\right|\d z.
\end{align*}
Since $u_{k}$ is a weak supersolution, the same arguments as above
but testing this time with $(\delta+w_{jk})\theta$ yield the analogous
estimate
\begin{align*}
\int_{\left\{ \left|u_{j}-u_{k}\right|<\delta\right\} } & -\theta\left|Du_{k}\right|^{p-2}Du_{k}\cdot\left(Du_{j}-Du_{k}\right)\d z\\
\leq & \int_{\Xi}2\delta\left|Du_{k}\right|^{p-1}\left|D\theta\right|+2\delta\theta\left|f(Du_{k})\right|-u_{k}\partial_{t}\left(w_{jk}\theta\right)+2\delta\left|u_{k}\right|\left|\partial_{t}\theta\right|\d z.
\end{align*}
Summing up these two inequalities we arrive at
\begin{align}
\int_{\left\{ \left|u_{j}-u_{k}\right|<\delta\right\} } & \theta\left(\left|Du_{j}\right|^{p-2}Du_{j}-\left|Du_{k}\right|^{p-2}Du_{k}\right)\cdot\left(Du_{j}-Du_{k}\right)\d z\nonumber \\
\leq & 2\delta\int_{\Xi}\left|D\theta\right|\left(\left|Du_{j}\right|^{p-1}+\left|Du_{k}\right|^{p-1}\right)\d z+2\delta\int_{\Xi}\theta\left(\left|f(Du_{j})\right|+\left|f(Du_{k})\right|\right)\d z\nonumber \\
 & +\int_{\Xi}(u_{j}-u_{k})\partial_{t}\left(w_{jk}\theta\right)\d z+2\delta\int_{\Xi}\left(\left|u_{j}\right|+\left|u_{k}\right|\right)\left|\partial_{t}\theta\right|\d z\nonumber \\
=: & I_{1}+I_{2}+I_{3}+I_{4}.\label{eq:lr conv 1}
\end{align}
We proceed to estimate these integrals. Denoting $\smash{M:=\sup_{j}\left\Vert u_{j}\right\Vert _{L^{\infty}(\supp\theta)}<\infty}$,
we have by the Caccioppoli's inequality Lemma \ref{lem:Caccioppoli}
\begin{equation}
\sup_{j}\int_{\supp\theta}\left|Du_{j}\right|^{p}\d z\leq C(p,\beta,C_{f},\theta,M).\label{eq:lr conv 2}
\end{equation}
The estimate (\ref{eq:lr conv 2}) and Hölder's inequality imply that
\[
I_{1}\leq\delta C(p,\beta,C_{f},\theta,M).
\]
To estimate $I_{2}$, we also use the growth condition (\ref{eq:gcnd})
and the assumption $\beta<p$. We get
\[
I_{2}\leq2\delta\int_{\Xi}\theta C_{f}(2+\left|Du_{j}\right|^{\beta}+\left|Du_{k}\right|^{\beta})\d z\leq\delta C(p,\beta,C_{f},\theta,M).
\]
The integral $I_{3}$ is estimated using integration by parts and
that $\left|w_{jk}\right|\leq\delta$ 
\begin{align*}
I_{3}= & \int_{\Xi}\theta(u_{j}-u_{k})\partial_{t}\left(w_{jk}\right)+\left(u_{j}-u_{k}\right)w_{jk}\partial_{t}\theta\d z=\int_{\Xi}\frac{1}{2}\theta\partial_{t}w_{jk}^{2}+\left(u_{j}-u_{k}\right)w_{jk}\partial_{t}\theta\d z\\
= & \int_{\Xi}-\frac{1}{2}w_{jk}^{2}\partial_{t}\theta+(u_{j}-u_{k})w_{jk}\partial_{t}\theta\d z\leq\delta C(\theta,M).
\end{align*}
For the last integral we have directly $I_{4}\leq\delta C(\theta,M)$.
Combining these estimates with (\ref{eq:lr conv 1}) we arrive at
\begin{equation}
\int_{\left\{ \left|u_{j}-u_{k}\right|<\delta\right\} }\theta\left(\left|Du_{j}\right|^{p-2}Du_{j}-\left|Du_{k}\right|^{p-2}Du_{k}\right)\cdot\left(Du_{j}-Du_{k}\right)\d z\leq\delta C_{0},\label{eq:lr conv 3}
\end{equation}
where $C_{0}=C(p,\beta,C_{f},\theta,M)$. If $1<p<2$, Hölder's inequality
and the algebraic inequality (\ref{eq:algebraic ineq 1<p<2}) give
the estimate (recall that $1<r<p$ and $\theta\equiv1$ in $U$)
\begin{align*}
 & \int_{U\cap\left\{ \left|u_{j}-u_{k}\right|<\delta\right\} }\left|Du_{j}-Du_{k}\right|^{r}\d z\\
 & \ \leq\bigg(\int_{U\cap\left\{ \left|u_{j}-u_{k}\right|<\delta\right\} }\left(1+\left|Du_{j}\right|^{2}+\left|Du_{k}\right|^{2}\right)^{\frac{r\left(2-p\right)}{2\left(2-r\right)}}\d z\bigg)^{\frac{2-r}{2}}\\
 & \ \ \ \ \ \ \cdot\bigg(\int_{U\cap\left\{ \left|u_{j}-u_{k}\right|<\delta\right\} }\frac{\left|Du_{j}-Du_{k}\right|^{2}}{\left(1+\left|Du_{j}\right|^{2}+\left|Du_{k}\right|^{2}\right)^{\frac{2-p}{2}}}\d z\bigg)^{\frac{r}{2}}\\
 & \ \leq C(p,\beta,r,C_{f},\theta,M)\\
 & \ \ \ \ \ \ \cdot\bigg(\int_{\left\{ \left|u_{j}-u_{k}\right|<\delta\right\} }\theta\left(\left|Du_{j}\right|^{p-2}Du_{j}-\left|Du_{k}\right|^{p-2}Du_{k}\right)\cdot\left(Du_{j}-Du_{k}\right)\d z\bigg)^{\frac{r}{2}},
\end{align*}
where in the last inequality we also used (\ref{eq:lr conv 2}) with
the knowledge $\frac{r(2-p)}{(2-r)}\leq\frac{p\left(2-p\right)}{2-p}=p.$
\\
If $p\geq2$, Hölder's inequality and the algebraic inequality (\ref{eq:algebraic ineq p>2})
imply
\begin{align*}
 & \int_{U\cap\left\{ \left|u_{j}-u_{k}\right|<\delta\right\} }\left|Du_{j}-Du_{k}\right|^{r}\d z\\
 & \ \leq\left(\int_{\Xi}1\d z\right)^{\frac{p-r}{p}}\bigg(\int_{U\cap\left\{ \left|u_{j}-u_{k}\right|<\delta\right\} }\left|Du_{j}-Du_{k}\right|^{p}\d z\bigg)^{\frac{r}{p}}\\
 & \ \leq C(p,r)\bigg(\int_{\left\{ \left|u_{j}-u_{k}\right|<\delta\right\} }\theta\left(\left|Du_{j}\right|^{p-2}Du_{j}-\left|Du_{k}\right|^{p-2}Du_{k}\right)\cdot\left(Du_{j}-Du_{k}\right)\d z\bigg)^{\frac{r}{p}}.
\end{align*}
Hence (\ref{eq:lr conv 3}) leads to
\[
\int_{U\cap\left\{ \left|u_{j}-u_{k}\right|<\delta\right\} }\left|Du_{j}-Du_{k}\right|^{r}\d z\leq\delta^{\frac{r}{\max(2,p)}}C(p,\beta,r,C_{f},\theta,M).
\]
On the other hand, Hölder's and Tchebysheff's inequalities with (\ref{eq:lr conv 2})
imply
\begin{align*}
 & \int_{U\cap\left\{ \left|u_{j}-u_{k}\right|\geq\delta\right\} }\left|Du_{j}-Du_{k}\right|^{r}\d z\\
 & \ \leq\left|U\cap\left\{ \left|u_{j}-u_{k}\right|\geq\delta\right\} \right|^{\frac{p-r}{p}}\bigg(\int_{U\cap\left\{ \left|u_{j}-u_{k}\right|\geq\delta\right\} }\left|Du_{j}-Du_{k}\right|^{p}\d z\bigg)^{\frac{r}{p}}\\
 & \ \leq\delta^{r-p}\left\Vert u_{j}-u_{k}\right\Vert _{L^{p}(U)}^{p-r}C(p,\beta,r,C_{f},\theta,M).
\end{align*}
So we arrive at
\[
\int_{U}\left|Du_{j}-Du_{k}\right|^{r}\d z\leq(\delta^{\frac{r}{\max(2,p)}}+\delta^{r-p}\left\Vert u_{j}-u_{k}\right\Vert _{L^{p}(U)}^{p-r})C(p,\beta,r,C_{f},\theta,M).
\]
Taking first small $\delta>0$ and then large $j,k$, we can make
the right-hand side arbitrarily small.
\end{proof}
Now we are ready to prove the main result of this section which states
that bounded viscosity supersolutions are weak supersolutions.
\begin{thm}
Let $1<p<\infty$. Let $u$ be a bounded viscosity supersolution to
(\ref{eq:p-para f}) in $\Xi$. Then $u$ is a weak supersolution
to (\ref{eq:p-para f}) in $\Xi$.
\end{thm}
\begin{proof}
Fix a non-negative test function $\smash{\varphi\in C_{0}^{\infty}(\Xi)}$
and take an open cylinder $\smash{\Omega_{t_{1},t_{2}}\Subset\Xi}$
such that $\smash{\supp\varphi\Subset\Omega_{t_{1},t_{2}}}$. Let
$\varepsilon>0$ be so small that $\smash{\Omega_{t_{1},t_{2}}\Subset\Xi_{\varepsilon}}$.
Then Lemma \ref{lem:inf conv is weak} implies that $\smash{u_{\varepsilon}}$
is a weak supersolution to (\ref{eq:p-para f}) in $\smash{\Xi_{\varepsilon}}$.
Therefore by the Caccioppoli's inequality Lemma \ref{lem:Caccioppoli},
$\smash{Du_{\varepsilon}}$ is bounded in $\smash{L^{p}(\Omega_{t_{1},t_{2}})}$.
Hence $\smash{Du_{\varepsilon}}$ converges weakly in $\smash{L^{p}(\Omega_{t_{1},t_{2}})}$
up to a subsequence. Since also $\smash{u_{\varepsilon}\rightarrow u}$
in $\smash{L^{\infty}(\Omega_{t_{1},t_{2}})}$, it follows that $\smash{u\in L^{p}(t_{1},t_{2};W^{1,p}(\Omega))}$. 

Since $\smash{u_{\varepsilon}}$ is a weak supersolution, it remains
to show that up to a subsequence
\begin{equation}
\lim_{\varepsilon\rightarrow0}\int_{\Omega_{t_{1},t_{2}}}u_{\varepsilon}\partial_{t}\varphi+\left|Du_{\varepsilon}\right|^{p-2}Du_{\varepsilon}\cdot D\varphi\d z=\int_{\Omega_{t_{1},t_{2}}}u\partial_{t}\varphi+\left|Du\right|^{p-2}Du\cdot D\varphi\d z\label{eq:visc is weak 2}
\end{equation}
and
\begin{equation}
\lim_{\varepsilon\rightarrow0}\int_{\Omega_{t_{1},t_{2}}}\varphi f(Du_{\varepsilon})\d z=\int_{\Omega_{t_{1},t_{2}}}\varphi f(Du)\d z.\ \ \ \ \ \label{eq:visc is weak 3}
\end{equation}
Since $\smash{u_{\varepsilon}\rightarrow u}$ in $\smash{L^{\infty}(\Omega_{t_{1},t_{2}})}$
and $\smash{Du_{\varepsilon}\rightarrow Du}$ in $\smash{L^{r}(\Omega_{t_{1},t_{2}})}$
for any $\smash{1<r<p}$ by Lemma \ref{lem:lr conv}, the claim (\ref{eq:visc is weak 2})
follows by applying the vector inequality (see \cite[p95-96]{lindqvist_plaplace})
\[
\left|\left|a\right|^{p-2}a-\left|b\right|^{p-2}b\right|\leq\begin{cases}
2^{2-p}\left|a-b\right|^{p-1} & \text{when }p<2,\\
2^{-1}\left(\left|a\right|^{p-2}+\left|b\right|^{p-2}\right)\left|a-b\right| & \text{when }p\geq2.
\end{cases}
\]
\begin{note}Indeed, if $1<p<2$ , we have
\begin{align*}
\int_{\Omega_{t_{1},t_{2}}^{\prime}}\left|Du_{\varepsilon}-Du\right|^{p-1}\d z\leq & \left|\Omega_{t_{1}t_{2}}^{\prime}\right|^{\frac{r-\left(p-1\right)}{r}}\bigg(\int_{\Omega_{t_{1},t_{2}}^{\prime}}\left|Du_{\varepsilon}-Du\right|^{r}\d z\bigg)^{\frac{p-1}{r}}
\end{align*}
for any $1<r<p$, and if $p>2$, then
\begin{align*}
\int_{\Omega_{t_{1},t_{2}}^{\prime}} & \left(\left|Du_{\varepsilon}\right|^{p-2}+\left|Du\right|^{p-2}\right)\left|Du_{\varepsilon}-Du\right|\d z\\
\leq & \bigg(\bigg(\int_{\Omega_{t_{1},t_{2}}^{\prime}}\left|Du_{\varepsilon}\right|^{p}\d z\bigg)^{\frac{p-2}{p}}+\bigg(\int_{\Omega_{t_{1},t_{2}}^{\prime}}\left|Du_{\varepsilon}\right|^{p}\d z\bigg)^{\frac{p-2}{p}}\bigg)\bigg(\int_{\Omega_{t_{1},t_{2}}^{\prime}}\left|Du_{\varepsilon}-Du\right|^{\frac{p}{2}}\d z\bigg)^{\frac{2}{p}}.
\end{align*}
\end{note}To show (\ref{eq:visc is weak 3}), let $M\geq1$ and write
using the growth condition (\ref{eq:gcnd})
\begin{align*}
\int_{\Omega_{t_{1},t_{2}}} & \left|f(Du_{\varepsilon})-f(Du)\right|\d z\\
\leq & \int_{\left\{ \left|Du_{\varepsilon}\right|<M\right\} }\left|f(Du_{\varepsilon})-f(Du)\right|\d z+\int_{\left\{ \left|Du_{\varepsilon}\right|\geq M\right\} }C_{f}(2+\left|Du_{\varepsilon}\right|^{\beta}+\left|Du\right|^{\beta})\d z\\
=: & I_{1}+I_{2}.
\end{align*}
Then by Hölder's inequality
\begin{align*}
I_{2} & =C_{f}\int_{\left\{ \left|Du_{\varepsilon}\right|\geq M\right\} }\frac{2\left|Du_{\varepsilon}\right|^{p}}{\left|Du_{\varepsilon}\right|^{p}}+\frac{\left|Du_{\varepsilon}\right|^{p}}{\left|Du_{\varepsilon}\right|^{p-\beta}}+\frac{\left|Du\right|^{\beta}\left|Du_{\varepsilon}\right|^{p-\beta}}{\left|Du_{\varepsilon}\right|^{p-\beta}}\d z\\
 & \leq C_{f}\bigg(\frac{2}{M^{p}}+\frac{1}{M^{p-\beta}}\bigg)\left\Vert Du_{\varepsilon}\right\Vert _{L^{p}(\Omega_{t_{1},t_{2}})}^{p}+C_{f}\frac{1}{M^{p-\beta}}\left\Vert Du\right\Vert _{L^{p}(\Omega_{t_{1},t_{2}})}^{\beta}\left\Vert Du_{\varepsilon}\right\Vert _{L^{p}(\Omega_{t_{1},t_{2}})}^{p-\beta}\\
 & \leq\frac{1}{M^{p-\beta}}C(p,\beta,C_{f},\left\Vert Du\right\Vert _{L^{p}(\Omega_{t_{1},t_{2}})},\sup_{\varepsilon}\left\Vert Du_{\varepsilon}\right\Vert _{L^{p}(\Omega_{t_{1},t_{2}})}).
\end{align*}
On the other hand, we have $\smash{\left|f(Du_{\varepsilon})-f(Du)\right|\rightarrow0}$
a.e.\ in $\smash{\Omega_{t_{1},t_{2}}}$ up to a subsequence and
the integrand in $\smash{I_{1}}$ is dominated by an integrable function
since the growth condition (\ref{eq:gcnd}) implies
\[
\left|f(Du_{\varepsilon})-f(Du)\right|\leq C_{f}(2+\left|M\right|^{\beta}+\left|Du\right|^{\beta})\quad\text{when}\quad\left|Du_{\varepsilon}\right|<M.
\]
Hence, for any $M\geq1$, we have $\smash{I_{1}\rightarrow0}$ as
$\varepsilon\rightarrow0$ by the dominated convergence theorem. By
taking first large $M\geq1$ and then small $\varepsilon>0$, we can
make $\smash{I_{1}+I_{2}}$ arbitrarily small.
\end{proof}
The rest of this section is devoted to the properties of the inf-convolution.
The facts in the following lemma are well known, see e.g. \cite{userguide},
\cite{newequivalence}, \cite{nikos} or \cite{ParvVaz}.
\begin{lem}
\renewcommand{\labelenumi}{(\roman{enumi})}\label{lem:inf conv prop}Assume
that $u:\Xi\rightarrow\mathbb{R}$ is lower semicontinuous and bounded.
Then $u_{\varepsilon}$ has the following properties.
\begin{enumerate}
\item We have $u_{\varepsilon}\leq u$ in $\Xi$ and $u_{\varepsilon}\rightarrow u$
locally uniformly as $\varepsilon\rightarrow0$.
\item Denote $r(\varepsilon):=\left(q\varepsilon^{q-1}\osc_{\Xi}u\right)^{\frac{1}{q}}$
, $t(\varepsilon):=\left(2\varepsilon\osc_{\Xi}u\right)^{\frac{1}{2}}$.
For $(x,t)\in\mathbb{R}^{N+1}$, set 
\begin{align*}
\Xi_{\varepsilon}:= & \left\{ (x,t)\in\Xi:B_{r(\varepsilon)}(x)\times(t-t(\varepsilon),t+t(\varepsilon))\Subset\Xi\right\} .
\end{align*}
Then for any $(x,t)\in\Xi_{\varepsilon}$ there exists $(x_{\varepsilon},t_{\varepsilon})\in B_{r(\varepsilon)}(x)\times(t-t(\varepsilon),t+t(\varepsilon))$
such that 
\[
u_{\varepsilon}(x,t)=u(x_{\varepsilon},t_{\varepsilon})+\frac{\left|x-x_{\varepsilon}\right|^{q}}{q\varepsilon^{q-1}}+\frac{\left|t-t_{\varepsilon}\right|^{2}}{2\varepsilon}.
\]
\item The function $u_{\varepsilon}$ is semi-concave in $\Xi_{\varepsilon}$
with a semi-concavity constant depending only on $u$, $q$ and $\varepsilon$.
\item Assume that $u_{\varepsilon}$ is differentiable in time and twice
differentiable in space at $(x,t)\in\Xi_{\varepsilon}$. Then 
\begin{align*}
\partial_{t}u_{\varepsilon}(x,t)= & \frac{t-t_{\varepsilon}}{\varepsilon},\\
Du_{\varepsilon}(x,t)= & \left(x-x_{\varepsilon}\right)\frac{\left|x-x_{\varepsilon}\right|^{q-2}}{\varepsilon^{q-1}},\\
D^{2}u_{\varepsilon}(x,t)\leq & \frac{q-1}{\varepsilon}\left|Du_{\varepsilon}\right|^{\frac{q-1}{q-2}}I.
\end{align*}
\end{enumerate}
\end{lem}
Next we show that the inf-convolution of a viscosity supersolution
to (\ref{eq:p-para f}) is still a supersolution in the smaller domain
$\smash{\Xi_{\varepsilon}}$. Since the inf-convolution is ``flat
enough'', that is, since $\smash{q>p/(p-1)}$, the inf-convolution
essentially cancels the singularity of the $\smash{p}$-Laplace operator.
This allows us to extract information on the time derivative at those
points of differentiability where $\smash{Du_{\varepsilon}}$ vanishes.
\begin{lem}
\label{lem:infconv is viscsup}Let $1<p<\infty$. Let $u$ be a viscosity
supersolution to (\ref{eq:p-para f}) in $\Xi$. Then the inf-convolution
$u_{\varepsilon}$ is also a viscosity supersolution to (\ref{eq:p-para f})
in $\Xi_{\varepsilon}$. 

Moreover, if $u_{\varepsilon}$ is differentiable in time and twice
differentiable in space at $(x,t)\in\Xi_{\varepsilon}$ and $Du_{\varepsilon}(x,t)=0$,
then $\partial_{t}u_{\varepsilon}(x,t)-f(0)\geq0$.
\end{lem}
\begin{proof}
Assume that $\varphi$ touches $u_{\varepsilon}$ from below at $(x,t)\in\Xi_{\varepsilon}$.
Let $(x_{\varepsilon},t_{\varepsilon})$ be like in the property (ii)
of Lemma \ref{lem:inf conv prop}. Then
\begin{align}
\varphi(x,t) & =u_{\varepsilon}(x,t)=u(x_{\varepsilon},t_{\varepsilon})+\frac{\left|x-x_{\varepsilon}\right|^{q}}{q\varepsilon^{q-1}}+\frac{\left|t-t_{\varepsilon}\right|^{2}}{2\varepsilon},\label{eq:infconv is viscsup1}\\
\varphi(y,\tau) & \leq u_{\varepsilon}(y,\tau)\leq u(z,s)+\frac{\left|y-z\right|^{q}}{q\varepsilon^{q-1}}+\frac{\left|\tau-s\right|^{2}}{2\varepsilon}\text{ for all }(y,\tau),(z,s)\in\Xi.\label{eq:infconv is viscsup2}
\end{align}
Set
\[
\psi(z,s):=\varphi(z+x-x_{\varepsilon},s+t-t_{\varepsilon})-\frac{\left|x-x_{\varepsilon}\right|^{q}}{q\varepsilon^{q-1}}-\frac{\left|t-t_{\varepsilon}\right|^{2}}{2\varepsilon}.
\]
Then $\psi$ touches $u$ from below at $(x_{\varepsilon},t_{\varepsilon})$
since by (\ref{eq:infconv is viscsup1})
\begin{align*}
\psi(x_{\varepsilon,}t_{\varepsilon})= & \varphi(x,t)-\frac{\left|x-x_{\varepsilon}\right|^{q}}{q\varepsilon^{q-1}}-\frac{\left|t-t_{\varepsilon}\right|^{2}}{2\varepsilon}=u(x_{\varepsilon},t_{\varepsilon})
\end{align*}
and selecting $(y,\tau)=(z+x-x_{\varepsilon},s+t-t_{\varepsilon})$
in (\ref{eq:infconv is viscsup2}) gives
\begin{align*}
\psi(z,s)= & \varphi(z+x-x_{\varepsilon},s+t-t_{\varepsilon})-\frac{\left|x-x_{\varepsilon}\right|^{q}}{q\varepsilon^{q-1}}-\frac{\left|t-t_{\varepsilon}\right|^{2}}{2\varepsilon}\leq u(z,s).
\end{align*}
Since $u$ is a viscosity supersolution, it follows that
\begin{align*}
0\leq & \limsup_{\substack{\substack{(z,s)\rightarrow(x_{\varepsilon},t_{\varepsilon})\\
z\not=x_{\varepsilon}
}
}
}\left(\partial_{s}\psi(z,s)-\Delta_{p}\psi(z,s)-f(D\psi(z,s))\right)\\
= & \limsup_{\substack{\substack{(z,s)\rightarrow(x,t)\\
z\not=x
}
}
}\left(\partial_{s}\varphi(z,s)-\Delta_{p}\varphi(z,s)-f(D\varphi(z,s))\right),
\end{align*}
and the first claim is proven. To prove the second claim, assume that
$u_{\varepsilon}$ is differentiable in time and twice differentiable
in space at $(x,t)\in\Xi_{\varepsilon}$ and $Du_{\varepsilon}(x,t)=0$.
By the property (iv) in Lemma \ref{lem:inf conv prop}, we have $x=x_{\varepsilon}$,
so that
\[
u_{\varepsilon}(x,t)=u(x,t_{\varepsilon})+\frac{\left|t-t_{\varepsilon}\right|^{2}}{2\varepsilon}.
\]
Hence by the definition of inf-convolution 
\[
u(y,s)+\frac{\left|x-y\right|^{q}}{q\varepsilon^{q-1}}+\frac{\left|t-s\right|^{2}}{2\varepsilon}\geq u_{\varepsilon}(x,t)=u(x,t_{\varepsilon})+\frac{\left|t-t_{\varepsilon}\right|^{2}}{2\varepsilon}\text{ for all }(y,s)\in\Xi.
\]
Arranging the terms as
\[
u(y,s)\geq u(x,t_{\varepsilon})-\frac{\left|x-y\right|^{q}}{q\varepsilon^{q-1}}-\frac{\left|t-s\right|^{2}}{2\varepsilon}+\frac{\left|t-t_{\varepsilon}\right|^{2}}{2\varepsilon}=:\phi(y,s),
\]
we see that the function $\phi$ touches $u$ from below at $(x,t_{\varepsilon})$.
Since $u$ is a viscosity supersolution and $D\phi(y,s)\not=0$ when
$y\not=x$, we have
\[
\limsup_{\substack{\substack{(y,s)\rightarrow(x,t_{\varepsilon})\\
y\not=x
}
}
}\left(\partial_{s}\phi(y,s)-\Delta_{p}\phi(y,s)-f(D\phi(y,s))\right)\geq0.
\]
On the other hand, since $q>p/(p-1)$, we have $\Delta_{p}\phi(y,s)\rightarrow0$
as $y\rightarrow x$. Hence we get
\[
0\leq\partial_{s}\phi(x,t_{\varepsilon})-f(0)=\frac{t-t_{\varepsilon}}{\varepsilon}-f(0)=\partial_{t}u_{\varepsilon}(x,t)-f(0),
\]
where the last equality follows from the property (iv) in Lemma \ref{lem:inf conv prop}.
\end{proof}
\begin{rem}
\label{rem:inf_remark}Semi-concavity implies that the inf-convolution
$\smash{u_{\varepsilon}}$ is locally Lipschitz in $\smash{\Xi_{\varepsilon}}$
(see \cite[p267]{measuretheoryevans}). Therefore $\smash{u_{\varepsilon}}$
is differentiable almost everywhere in $\smash{\Xi_{\varepsilon}}$,
$\smash{\partial_{t}u_{\varepsilon}\in L_{loc}^{\infty}(\Xi_{\varepsilon})}$
and $\smash{u_{\varepsilon}\in L^{\infty}(t_{1},t_{2};W^{1,\infty}(\Omega))}$
for any $\smash{\Omega_{t_{1},t_{2}}\Subset\Xi_{\varepsilon}}$ (see
\cite[p266]{measuretheoryevans}).

Moreover, since the function $\phi(x,t):=u_{\varepsilon}(x,t)-C(q,\varepsilon,u)(\left|x\right|^{2}+\left|t\right|^{2})$
is concave, Alexandrov's theorem implies that $\smash{u_{\varepsilon}}$
is twice differentiable almost everywhere in $\smash{\Xi_{\varepsilon}}.$
Furthermore, the proof of Alexandrov's theorem in \cite[p273]{measuretheoryevans}
establishes that if $\smash{\phi_{j}}$ is the standard mollification
of $\phi$, then $\smash{D^{2}\phi_{j}\rightarrow D^{2}\phi}$ almost
everywhere in $\smash{\Xi_{\varepsilon}}$.
\end{rem}

\section{Lower semicontinuity of supersolutions}

We show the lower semicontinuity of weak supersolutions when $p\geq2$
and the function $f\in C(\mathbb{R}^{N})$ satisfies that $f(0)=0$
as well as the stronger growth condition
\begin{equation}
\left|f(\xi)\right|\leq C_{f}\left(1+\left|\xi\right|^{p-1}\right).\tag{{G2}}\label{eq:gcnd2}
\end{equation}
Our proof follows the method of Kuusi \cite{kuusi09}, but the first-order
term causes some modifications. In particular, our essential supremum
estimate is slightly different, see Theorem \ref{thm:essupest} and
the brief discussion before it. The assumption $f(0)=0$ is used to
ensure that the positive part $u_{+}$ of a subsolution is still a
subsolution.

We begin by proving estimates for the essential supremum of a subsolution
using the Moser's iteration technique. We first need the following
Caccioppoli's inequalities.
\begin{lem}[Caccioppoli's inequalities]
\label{lem:caccioppoli0} Assume that $p\geq2$ and that (\ref{eq:gcnd2})
holds. Suppose that $u$ is a non-negative weak subsolution to (\ref{eq:p-para f})
in $\Omega_{t_{1},t_{2}}$ and $u\in L^{p-1+\lambda}(\Omega_{t_{1},t_{2}})$
for some $\lambda\geq1$. Then there exists a constant $C=C(p,C_{f})$
that satisfies the estimates
\begin{align*}
\essup_{t_{1}<\tau<t_{2}} & \int_{\Omega}u^{1+\lambda}(x,\tau)\zeta^{p}(x,\tau)\d x\\
\leq & C\int_{\Omega_{t_{1},t_{2}}}\lambda u^{p-1+\lambda}\left|D\zeta\right|^{p}+u^{1+\lambda}\left|\partial_{t}\zeta\right|\zeta^{p-1}+\lambda\left(u^{\lambda}+u^{p-1+\lambda}\right)\zeta^{p}\d z
\end{align*}
and
\begin{align*}
\int_{\Omega_{t_{1},t_{2}}} & \left|D(u^{\frac{p-1+\lambda}{p}}\zeta)\right|^{p}\d z\\
\leq & C\int_{\Omega_{t_{1},t_{2}}}\lambda^{p}u^{p-1+\lambda}\left|D\zeta\right|^{p}+\lambda^{p-1}u^{1+\lambda}\left|\partial_{t}\zeta\right|\zeta^{p-1}+\lambda^{p}\left(u^{\lambda}+u^{p-1+\lambda}\right)\zeta^{p}\d z
\end{align*}
for all non-negative $\zeta\in C^{\infty}(\Omega\times[t_{1},t_{2}])$
such that $\supp\zeta(\cdot,t)\Subset\Omega$ and $\zeta(x,t_{1})=0$. 
\end{lem}
\begin{proof}
We test the regularized equation in Lemma (\ref{lem:time conv lemma})
with $\varphi:=\min(u^{\epsilon},k)^{\lambda-1}u^{\epsilon}\zeta^{p}\eta$,
where $\eta$ is the following cut-off function
\[
\eta(t)=\begin{cases}
0, & t\in(t_{1},s-h),\\
(t-s+h)/2h, & t\in[s-h,s+h],\\
1, & t\in(s+h,\tau-h),\\
(-t+\tau+h)/2h, & t\in[\tau-h,\tau+h],\\
0, & t\in(\tau+h,t_{2}),
\end{cases}
\]
and $t_{1}<s<\tau<t_{2}$, $h>0$. We denote $g(l):=\int_{0}^{l}\min(r,k)^{\lambda-1}r\d r$.
Then integration by parts and Lebesgue's differentiation theorem yield
for a.e.\ $s,\tau\in(t_{1},t_{2})$
\begin{align*}
\int_{\Omega_{t_{1},t_{2}}} & \partial_{t}(u^{\epsilon})\min(u^{\epsilon},k)^{\lambda-1}u^{\epsilon}\zeta^{p}\eta\d z\\
= & \int_{\Omega_{t_{1},t_{2}}}\partial_{t}g(u^{\epsilon})\zeta^{p}\eta\d z\\
= & \int_{\Omega_{t_{1},t_{2}}}-\eta g(u^{\epsilon})\partial_{t}(\zeta^{p})-\zeta^{p}g(u^{\epsilon})\partial_{t}\eta\d z\\
\underset{\epsilon\rightarrow0,h\rightarrow0}{\rightarrow} & \int_{\Omega_{s,\tau}}-g(u)\partial_{t}(\zeta^{p})\d z-\int_{\Omega}\zeta^{p}(x,s)g(u(x,s))\d x+\int_{\Omega}\zeta^{p}(x,\tau)g(u(x,\tau))\d x.
\end{align*}
Letting $s\rightarrow t_{1}$ and observing that the other terms of
(\ref{eq:p-para f reg}) converge as well, we obtain for a.e.\ $\tau\in(t_{1},t_{2})$
that
\begin{align*}
\int_{\Omega} & g(u(x,\tau))\zeta^{p}(x,\tau)\d x\\
 & \leq\int_{\Omega_{t_{1},\tau}}g(u)\partial_{t}(\zeta^{p})-\left|Du\right|^{p-2}Du\cdot D(u_{k}^{\lambda-1}u\zeta^{p})+u_{k}^{\lambda-1}u\zeta^{p}f(Du)\d z,
\end{align*}
where we have denoted $u_{k}:=\min(u,k)$. Since
\[
Du_{k}^{\lambda-1}=\chi_{\left\{ u<k\right\} }(\lambda-1)u^{\lambda-2}Du,
\]
we have by Young's inequality
\begin{align*}
-\left|Du\right|^{p-2}Du\cdot D(u_{k}^{\lambda-1}u\zeta^{p})\leq & -\zeta^{p}\left((\lambda-1)\chi_{\left\{ u<k\right\} }u^{\lambda-1}+u_{k}^{\lambda-1}\right)\left|Du\right|^{p}\\
 & +p\zeta^{p-1}u_{k}^{\lambda-1}u\left|Du\right|^{p-1}\left|D\zeta\right|\\
\leq & -\frac{1}{2}\zeta^{p}u_{k}^{\lambda-1}\left|Du\right|^{p}+C(p)u^{p-1+\lambda}\left|D\zeta\right|^{p}.
\end{align*}
Moreover, by the growth condition (\ref{eq:gcnd2}) and Young's inequality
\begin{align*}
u_{k}^{\lambda-1}u\zeta^{p}f(Du)\leq & C_{f}\zeta^{p}u_{k}^{\lambda-1}u+C_{f}\zeta^{p}u_{k}^{\lambda-1}u\left|Du\right|^{p-1}\\
\leq & C_{f}\zeta^{p}u^{\lambda-1}+C(p,C_{f})\zeta^{p}u^{p-1+\lambda}+\frac{1}{4}\zeta^{p}u_{k}^{\lambda-1}\left|Du\right|^{p}.
\end{align*}
Collecting the estimates, moving the terms with $Du$ to the left-hand
side and letting $k\rightarrow\infty$, we arrive at
\begin{align}
\lambda^{-1}\int_{\Omega} & u^{\lambda+1}\zeta^{p}(x,\tau)\d x+\int_{\Omega_{t_{1},\tau}}\frac{1}{4}\zeta^{p}u^{\lambda-1}\left|Du\right|^{p}\d z\nonumber \\
\leq & C(p,C_{f})\int_{\Omega_{t_{1},\tau}}\lambda^{-1}u^{\lambda+1}\left|\partial_{t}\zeta^{p}\right|+u^{p-1+\lambda}\left|D\zeta\right|^{p}+\zeta^{p}(u^{\lambda-1}+u^{p-1+\lambda})\d z.\label{eq: some est}
\end{align}
Since the integrals are positive, this yields the first inequality
of the lemma by taking essential supremum over $\tau$. The second
inequality follows from (\ref{eq: some est}) by using that
\begin{align*}
\int_{\Omega_{t_{1},t_{2}}}\left|D(u^{\frac{p-1+\lambda}{p}}\zeta)\right|^{p}\d z\leq & C(p)\int_{\Omega_{t_{1},t_{2}}}u^{p-1+\lambda}\left|D\zeta\right|^{p}+\lambda^{p}\zeta^{p}u^{\lambda-1}\left|Du\right|^{p}\d z.\qedhere
\end{align*}
\end{proof}
We first prove the following essential supremum estimate where we
assume that the subsolution is bounded away from zero.
\begin{lem}
\label{lem:essupest} Assume that $p\geq2$ and that (\ref{eq:gcnd2})
holds. Suppose that $u$ is a weak subsolution to (\ref{eq:p-para f})
in $\Xi$ and $B_{R}(x_{0})\times(t_{0}-T,t_{0})\Subset\Xi$ where
$R,T<1$ are such that
\begin{equation}
\frac{R^{p}}{T}\leq1\ \ \text{and}\ \ u\geq\left(\frac{R^{p}}{T}\right)^{\frac{1}{p-1}}.\label{eq:essupest bound}
\end{equation}
Then there exists a constant $C(N,p,C_{f})$ such that
\[
\essup_{B_{\sigma R(x_{0})\times(t_{0}-\sigma^{p}T,t_{0})}}u\leq C\left(\frac{T}{R^{p}}(1-\sigma)^{-N-p}\aint_{B_{R}(x_{0})\times(t_{0}-T,t_{0})}u^{p-2+\delta}\d z\right)^{1/\delta}
\]
for every $1/2\leq\sigma<1$ and $1<\delta<2$.
\end{lem}
\begin{proof}
Let $\sigma R\leq s<S<R$. For $j\in0,1,2,\ldots$, we set
\[
R_{j}:=S-\left(S-s\right)(1-2^{-j})
\]
and 
\[
U_{j}:=B_{j}\times\Gamma_{j}:=B_{R_{j}}(x_{0})\times(t_{0}-(R_{j}/S)^{p}T,t_{0}).
\]
We choose test functions $\varphi_{j}\in C^{\infty}(\overline{U_{j}})$
such that $\supp\varphi_{j}(\cdot,t)\Subset B_{R_{j}}(x_{0})$,
\[
0\leq\varphi_{j}\leq1,\ \ \varphi_{j}\equiv0\text{ on }\partial_{p}U_{j},\ \ \varphi_{j}\equiv1\text{ in }U_{j+1}
\]
and 
\[
\left|D\varphi_{j}\right|\leq\frac{C}{S-s}2^{j},\ \ \left|\partial_{t}\varphi_{j}\right|\leq\frac{R^{p}}{T}\frac{C}{(S-s)^{p}}2^{jp}.
\]
We set $\gamma:=1+p/N$ and
\[
\lambda_{j}:=2\gamma^{j}-1,j=0,1,2,\ldots.
\]
Assuming that we already know that $u\in L^{p-1+\lambda_{j}}(U_{j})$,
then we have by a parabolic Sobolev's inequality (see \cite[p7]{dibenedetto93})
\begin{align}
\int_{U_{j+1}} & u^{\kappa\alpha}\d z\leq\int_{U_{j}}\left(u^{\alpha/p}\varphi_{j}^{\beta/p}\right)^{\kappa p}\d z\nonumber \\
\leq & C(N,p)\int_{U_{j}}\left|D(u^{\alpha/p}\varphi_{j}^{\beta/p})\right|^{p}\d z\left(\essup_{\Gamma_{j}}\int_{B_{j}}\left(u^{\alpha/p}\varphi_{j}^{\beta/p}\right)^{\left(\kappa-1\right)N}\d x\right)^{p/N},\label{eq:essupest_sobolev}
\end{align}
where
\[
\alpha=p-1+\lambda_{j},\ \text{ }\kappa=1+\frac{p(1+\lambda_{j})}{N(p-1+\lambda_{j})},\ \text{ \ensuremath{\beta}=}\frac{p(p-1+\lambda_{j})}{1+\lambda_{j}}.
\]
The first estimate in Lemma \ref{lem:caccioppoli0} gives
\begin{align}
\essup_{\Gamma_{j}} & \int_{B_{j}}\left(u^{\alpha/p}\varphi_{j}^{\beta/p}\right)^{\left(\kappa-1\right)N}\d x=\essup_{\Gamma_{j}}\int_{B_{j}}u^{1+\lambda_{j}}\varphi_{j}^{p}\d x\nonumber \\
\leq & C\lambda_{j}\int_{U_{j}}u^{p-1+\lambda_{j}}\left|D\varphi_{j}\right|^{p}+u^{1+\lambda_{j}}\left|\partial_{t}\varphi_{j}\right|\varphi_{j}^{p-1}+\left(u^{\lambda_{j}}+u^{p-1+\lambda_{j}}\right)\varphi_{j}^{p}\d z.\label{eq:essupest_ca1}
\end{align}
Using the second estimate with $\zeta=\varphi_{j}^{\beta/p}$ we obtain
\begin{align}
\int_{U_{j}} & \left|D(u^{\alpha/p}\varphi_{j}^{\beta/p})\right|^{p}\d z\nonumber \\
\leq & C\lambda_{j}^{p}\int_{U_{j}}u^{p-1+\lambda_{j}}\left|D\varphi_{j}\right|^{p}+u^{1+\lambda_{j}}\left|\partial_{t}\varphi_{j}\right|\varphi_{j}^{p-1}+\left(u^{\lambda_{j}}+u^{p-1+\lambda_{j}}\right)\varphi_{j}^{p}\d z.\label{eq:essupest_ca2}
\end{align}
Combining (\ref{eq:essupest_sobolev}) with (\ref{eq:essupest_ca1})
and (\ref{eq:essupest_ca2}) we arrive at
\begin{equation}
\left(\int_{U_{j+1}}u^{\kappa\alpha}\d z\right)^{\frac{1}{\gamma}}\leq C\lambda_{j}^{p}\int_{U_{j}}\frac{2^{jp}}{\left(S-s\right)^{p}}u^{p-1+\lambda_{j}}+\frac{R^{p}2^{jp}}{T(S-s)^{p}}u^{1+\lambda_{j}}+u^{\lambda_{j}}\d z,\label{eq:preit est}
\end{equation}
where $\smash{\gamma=1+p/N}$. We wish to iterate this inequality,
but having multiple terms at the right-hand side is a problem. This
is where the assumption (\ref{eq:essupest bound}) comes into play.
Since $\smash{u\geq(R^{p}/T)^{1/(p-1)}}$, we have 
\[
u^{\lambda_{j}}=\left(\frac{1}{u}\right)^{p-1}u^{p-1+\lambda_{j}}\leq\left(\frac{T}{R^{p}}\right)^{\frac{p-1}{p-1}}u^{p-1+\lambda_{j}}\leq\frac{1}{(S-s)^{p}}u^{p-1+\lambda_{j}}
\]
and since $T/R^{p}\geq1$, we have also
\[
u^{1+\lambda_{j}}=\left(\frac{1}{u}\right)^{p-2}u^{p-1+\lambda_{j}}\leq\left(\frac{T}{R^{p}}\right)^{\frac{p-2}{p-1}}u^{p-1+\lambda_{j}}\leq\frac{T}{R^{p}}u^{p-1+\lambda_{j}}.
\]
Using these estimates it follows from (\ref{eq:preit est}) that
\begin{align}
\left(\int_{U_{j+1}}u^{\kappa\alpha}\d z\right)^{\frac{1}{\gamma}}\leq & \frac{C\lambda_{j}^{p}2^{jp}}{(S-s)^{p}}\int_{U_{j}}u^{p-1+\lambda_{j}}\d z.\label{eq:lalala}
\end{align}
Observe that
\[
\kappa\alpha=p-1+\lambda_{j}(1+p/N)+p/N=p-1+\lambda_{j+1}.
\]
Hence by denoting $Y:=C(S-s)^{-p}$, the inequality (\ref{eq:lalala})
becomes 
\[
\left(\int_{U_{j+1}}u^{p-1+\lambda_{j+1}}\d z\right)^{\frac{1}{\gamma}}\leq Y(2\gamma)^{jp}\int_{U_{j}}u^{p-1+\lambda_{j}}\d z.
\]
We iterate this inequality. When $j=0,$ it reads as
\[
\left(\int_{U_{1}}u^{p-1+\lambda_{1}}\d z\right)^{\frac{1}{\gamma}}\leq Y\int_{U_{0}}u^{p}\d z.
\]
Then, when $j=1$, we have
\begin{align*}
\left(\int_{U_{2}}u^{p-1+\lambda_{2}}\d z\right)^{\frac{1}{\gamma^{2}}}\leq & Y^{\frac{1}{\gamma}}(2\gamma)^{p\frac{1}{\gamma}}\left(\int_{U_{1}}u^{p-1+\lambda_{1}}\d z\right)^{\frac{1}{\gamma}}\leq Y^{1+\frac{1}{\gamma}}(2\gamma)^{p\frac{1}{\gamma}}\int_{U_{0}}u^{p}\d z.
\end{align*}
Continuing this way we arrive at
\begin{align*}
\left(\int_{U_{j+1}}u^{p-1+\lambda_{j+1}}\d z\right)^{\frac{1}{\gamma^{j+1}}}\leq & Y^{1+\frac{1}{\gamma}+\ldots+\frac{1}{\gamma^{j}}}(2\gamma)^{p(\frac{1}{\gamma}+\frac{2}{\gamma^{2}}+\ldots+\frac{j}{\gamma^{j}})}\int_{U_{0}}u^{p}\d z\\
\leq & CY^{\frac{N}{p}+1}\int_{U_{0}}u^{p}\d z,
\end{align*}
so that
\[
\left(\int_{U_{j+1}}u^{p-1+\lambda_{j+1}}\d z\right)^{\frac{1}{p-1+\lambda_{j+1}}}\leq\left(CY^{\frac{N}{p}+1}\int_{U_{0}}u^{p}\d z\right)^{\frac{\gamma^{j+1}}{p-1+\lambda_{j+1}}}.
\]
Since $\gamma^{j+1}/(p-1+\lambda_{j+1})\rightarrow1/2$ and $p-1+\lambda_{j+1}\rightarrow\infty$
as $j\rightarrow\infty$, we obtain that
\[
\essup_{Q(s)}u\leq C\left((S-s)^{-N-p}\int_{Q(S)}u^{p}\d z\right)^{1/2},
\]
where $Q(s)=B(x_{0},s)\times(t_{0}-(s/S)^{p}T,t_{0})$. By Young's
inequality we have for every $1<\delta<2$ that
\begin{align}
\essup_{Q(s)}u\leq & \left(\essup_{Q(S)}u^{2-\delta}(S-s)^{-N-p}\int_{Q(S)}u^{p-2+\delta}\d z\right)^{1/2}\nonumber \\
\leq & \frac{1}{2}\essup_{Q(S)}u+\left((S-s)^{-N-p}\int_{B_{R}(x_{0})\times(t_{0}-T,t_{0})}u^{p-2+\delta}\d z\right)^{1/\delta}.\label{eq:essupest aux}
\end{align}
A standard iteration argument such as \cite[Lemma 1.1]{giaquintaGiusti82}
now finishes the proof. Indeed, if $f:[T_{0},T_{1}]\rightarrow\mathbb{R}$
is a non-negative bounded function such that all $T_{0}\leq t\leq\tau\leq T_{1}$
satisfy
\begin{equation}
f(t)\leq\theta f(\tau)+(\tau-t)^{-\eta}A,\label{eq:essupest it assumption}
\end{equation}
where $A,\theta,\eta\geq0$ with $\theta<1$, then
\[
f(T_{0})\leq C(\eta,\theta)(T_{1}-T_{0})^{-\eta}A.
\]
Selecting $T_{0}:=\sigma R$, $T_{1}:=\left(\sigma R+R\right)/2$
and the other variables so that (\ref{eq:essupest aux}) implies (\ref{eq:essupest it assumption}),
we get the desired estimate.
\end{proof}
Next we consider the case where the non-negative subsolution is not
necessarily bounded away from zero. Observe that the estimate differs
from the usual estimate for the $p$-Laplacian because of the power
$1/(p-1)$ in the first term (cf. \cite[Theorem 4.1]{dibenedetto93}
or \cite[Theorem 3.4]{kuusi09}). However, we have the additional
assumption (\ref{eq:essupest cylinder cnd}).
\begin{thm}
\label{thm:essupest} Assume that $p\geq2$ and that (\ref{eq:gcnd2})
holds. Suppose that $u$ is a non-negative weak subsolution to (\ref{eq:p-para f})
in $\Xi$ and $B_{R}(x_{0})\times(t_{0}-T,t_{0})\Subset\Xi$ with
$R,T<1$ such that
\begin{equation}
\frac{R^{p}}{T}\leq1.\label{eq:essupest cylinder cnd}
\end{equation}
Then there exists a constant $C=C(N,p,C_{f},\delta)$ such that we
have the estimate
\[
\essup_{B(x_{0},R/2)\times(t_{0}-T/2^{p},t_{0})}u\leq C\left(\frac{R^{p}}{T}\right)^{\frac{1}{p-1}\cdot\frac{\delta-1}{\delta}}+C\left(\frac{T}{R^{p}}\aint_{t_{0}-T}^{t_{0}}\aint_{B_{R}(x_{0})}u^{p-2+\delta}\d x\d t\right)^{\frac{1}{\delta}}
\]
for all $1<\delta<2$.
\end{thm}
\begin{proof}
We denote
\[
\varLambda:=(1-\sigma)^{-N-p},\ \theta:=\left(\frac{R^{p}}{T}\right)^{\frac{1}{p-1}}.
\]
Using Lemma \ref{lem:essupest} on the subsolution $v:=\theta+u$
we get the estimate
\begin{align*}
\essup_{B_{\sigma R(x_{0})\times(t_{0}-\sigma^{p}T,t_{0})}}u\leq & C\left(\varLambda\frac{T}{R^{p}}\aint_{B_{R}(x_{0})\times(t_{0}-T,t_{0})}\left(\theta+u\right)^{p-2+\delta}\d z\right)^{\frac{1}{\delta}}\\
\leq & C\varLambda^{\frac{1}{\delta}}\left(\frac{T}{R^{p}}\theta^{p-2+\delta}\right)^{\frac{1}{\delta}}+C\varLambda^{\frac{1}{\delta}}\left(\frac{T}{R^{p}}\aint_{B_{R}(x_{0})\times(t_{0}-T,t_{0})}u^{p-2+\delta}\d z\right)^{\frac{1}{\delta}},
\end{align*}
where
\[
\frac{T}{R^{p}}\theta^{p-2+\delta}=T^{1-\frac{p-2+\delta}{p-1}}R^{-p+\frac{p(p-2+\delta)}{p-1}}=\left(T^{1-\delta}R^{p\left(\delta-1\right)}\right)^{\frac{1}{p-1}}=\left(\frac{R^{p}}{T}\right)^{\frac{\delta-1}{p-1}}.
\]
Taking $\sigma=1/2$ now yields the desired inequality.
\end{proof}
\begin{lem}
\label{lem:positive part is sub}Assume that $p\geq2$ and that $f(0)=0$.
Let $u$ be a weak subsolution to (\ref{eq:p-para f}) in $\Omega_{t_{1},t_{2}}$.
Then $u_{+}=\max(u,0)$ is also a weak subsolution.
\end{lem}
\begin{proof}
Fix a non-negative test function $\zeta\in C_{0}^{\infty}(\Omega_{t_{1},t_{2}})$.
We test the regularized equation in Lemma \ref{lem:time conv lemma}
with $\min\left\{ k(u^{\epsilon})_{+},1\right\} \zeta$. Then by similar
arguments as in the proof of Lemma \ref{lem:caccioppoli0} we get
the estimate 
\begin{align*}
\int_{\Omega_{t_{1},t_{2}}} & \min\left\{ ku_{+},1\right\} (-u\partial_{t}\zeta+\left|Du\right|^{p-2}Du\cdot D\zeta-\zeta f(Du))\d z\\
\leq & -\frac{1}{2k}\int_{\Omega_{t_{1},t_{2}}}\left(\min\left\{ ku_{+},1\right\} \right)^{2}\partial_{t}\zeta\d z-k\int_{\left\{ 0<ku<1\right\} }\zeta\left|Du\right|^{p}\d z.
\end{align*}
Letting $k\rightarrow\infty$ this implies
\[
\int_{\left\{ u>0\right\} }-u\partial_{t}\zeta+\left|Du\right|^{p-2}Du\cdot D\zeta-\zeta f(Du)\d z\leq0.
\]
Since $f(0)=0$ and $u_{+}\partial_{t}\zeta=0=Du_{+}$ a.e.$\ $in
$\left\{ u\leq0\right\} $, we get that
\[
\int_{\Omega_{t_{1},t_{2}}}-u_{+}\partial_{t}\zeta+\left|Du_{+}\right|^{p-2}Du_{+}\cdot D\zeta-\zeta f(Du_{+})\d z\leq0.\qedhere
\]
\end{proof}
\begin{thm}
\label{thm:semicont} Assume that $p\geq2$, (\ref{eq:gcnd2}) holds
and that $f(0)=0$. Suppose that $u$ is a weak supersolution to (\ref{eq:p-para f})
in $\Xi$. Let $u_{\ast}$ denote the lower semicontinuous regularization
of $u$, that is,
\begin{align*}
u_{\ast}(x,t):= & \essliminf_{(y,s)\rightarrow(x,t)}u(y,s):=\lim_{R\rightarrow0}\ \essinf_{B_{R}(x)\times(t-R^{p},t+R^{p})}u.
\end{align*}
Then $u=u_{\ast}$ almost everywhere.
\end{thm}
\begin{proof}
For all $M\in\mathbb{N}$, we define the cylinders
\[
Q_{R}^{M}(x,t):=B_{R}(x)\times(t-MR^{p},t+MR^{p}).
\]
We denote by $E_{M}$ the set of Lebesgue points with respect to the
basis $\{Q_{R}^{M}\}$, that is, 
\[
E_{M}:=\left\{ (x,t)\in\Xi:\lim_{R\rightarrow0}\aint_{Q_{R}^{M}(x,t)}\left|u(x,t)-u(y,s)\right|^{p-\frac{1}{2}}\d y\d s=0\right\} .
\]
Then $E_{M}\subset E_{M+1}$ so that 
\[
E:=\bigcap_{M\in\mathbb{N}}E_{M}=E_{1}.
\]
Moreover, we have $\left|E\right|=\left|\Xi\right|$, which follows
from \cite[p13]{Stein93} by a simple argument, see for example \cite[p54]{measuretheoryevans}.

We now claim that if $(x_{0},t_{0})\in E$, then
\begin{equation}
u(x_{0},t_{0})\leq\essliminf_{(x,t)\rightarrow(x_{0},t_{0})}u(x,t).\label{eq:semicont claim}
\end{equation}
We make the counter assumption
\[
u(x_{0},t_{0})-\essliminf_{(x,t)\rightarrow(x_{0},t_{0})}u(x,t)=\varepsilon>0.
\]
Let $R_{0}$ be a radius such that
\[
\left|\essliminf_{(x,t)\rightarrow(x_{0},t_{0})}u(x,t)-\essinf_{Q_{R}^{1}(x_{0},t_{0})}u\right|\leq\varepsilon/2
\]
for all $0<R\leq R_{0}$. For such $R$ we have
\begin{equation}
u(x_{0},t_{0})-\essinf_{Q_{R}^{1}(x_{0},t_{0})}u\geq\varepsilon/2.\label{eq:semicont 1}
\end{equation}
We set $v:=(u(x_{0},t_{0})-u)_{+}$. Since $(x_{0},t_{0})\in E$,
we find for any $M\in\mathbb{N}$ a radius $R_{1}=R_{1}(M)$ such
that
\begin{align}
\aint_{Q_{R_{1}}^{M}(x_{0},t_{0})} & v^{p-\frac{1}{2}}\d x\d t\leq\aint_{Q_{R_{1}}^{M}(x_{0},t_{0})}\left|u(x_{0},t_{0})-u\right|^{p-\frac{1}{2}}\d x\d t\leq\left(\frac{1}{M}\right)^{2}.\label{eq:semicont 2}
\end{align}
On the other hand, by Lemma \ref{lem:positive part is sub} the function
$v$ is a weak subsolution to
\[
\partial_{t}v+\Delta_{p}v-g(Dv)\leq0,
\]
where $g(\xi)=-f(-\xi)$. Observe also that the cylinder $Q_{R_{1}}^{M}(x_{0},t_{0})$
satisfies the condition (\ref{eq:essupest cylinder cnd}) since $R_{1}^{p}/(MR_{1}^{p})\leq1$.
Hence we may apply Theorem \ref{thm:essupest} with $\delta=3/2$
and then use (\ref{eq:semicont 2}) to get
\begin{align*}
\essup_{Q_{(R_{1})/2}^{M}(x_{0},t_{0})}v\leq & C\left(\frac{R_{1}^{p}}{R_{1}^{p}M}\right)^{\frac{1}{3\left(p-1\right)}}+C\left(\frac{R_{1}^{p}M}{R_{1}^{p}}\aint_{Q_{R_{1}}^{M}(x_{0},t_{0})}v^{p-\frac{1}{2}}\d x\d t\right)^{\frac{2}{3}}\\
\leq & \frac{C}{M^{3(p-1)}}+C\left(M\cdot\frac{1}{M^{2}}\right)^{\frac{2}{3}}\\
\leq & C\left(\frac{1}{M}\right)^{\frac{1}{3}}.
\end{align*}
Now we first fix $M$ so large that $C/M^{\frac{1}{3}}\leq\varepsilon/4$
and this will also fix $R_{1}$. Then we take $R\in(0,R_{0}]$ so
small that $Q_{R}^{1}(x_{0},t_{0})\subset Q_{(R_{1})/2}^{M}(x_{0},t_{0})$.
Then (\ref{eq:semicont 1}) leads to a contradiction since
\[
\varepsilon/4\geq\essup_{Q_{(R_{1})/2}^{M}(x_{0},t_{0})}v\geq\essup_{Q_{R}^{1}(x_{0},t_{0})}v\geq u(x_{0},t_{0})-\essinf_{Q_{R}^{1}(x_{0},t_{0})}u\geq\varepsilon/2.
\]
Hence (\ref{eq:semicont claim}) holds and we have
\[
u(x_{0},t_{0})\leq\essliminf_{(x,t)\rightarrow(x_{0},t_{0})}u(x,t)\leq\lim_{R\rightarrow0}\aint_{Q_{R}^{1}}u(x,t)\d x\d t=u(x_{0},t_{0}).
\]
Thus $u_{\ast}=u$ almost everywhere and it is easy to show that $u_{\ast}$
is lower semicontinuous.
\end{proof}
\bibliographystyle{alpha}

\Addresses
\end{document}